\documentclass[a4paper]{article}

\usepackage[utf8]{inputenc}
\usepackage{amsfonts,amsmath,amssymb,amsthm}
\usepackage[hidelinks]{hyperref}
\usepackage{enumerate}
\usepackage{bbm}
\usepackage{verbatim}
\usepackage{authblk}
\usepackage{times}
\usepackage{graphicx,float}
\usepackage{mdframed}
\usepackage{lipsum}
\usepackage{caption}

\usepackage{tikz,tkz-euclide,pgfplots}
\usetikzlibrary{calc,patterns}
\usetikzlibrary{arrows,shapes,positioning}
\usetikzlibrary{decorations.markings}
\tikzstyle arrowstyle=[scale=1]
\tikzstyle directed=[postaction={decorate, decoration={markings,
    mark=at position .6 with {\arrow[arrowstyle]{stealth}}}}]

\pgfplotsset{compat=1.17}

\DeclareMathOperator{\dens}{dens}

\DeclareMathOperator{\Var}{Var}

\DeclareMathOperator{\Cay}{Cay}
\newcommand{\moinsun}{^{-1}}
\newcommand{\bs}{\boldsymbol}
\newcommand{\tendvers}{\xrightarrow[\ell\to\infty]{}}
\def\Pr{\mathbf{Pr}}
\def\P{\mathbf{P}}
\newcommand\PrCond[2]{\Pr\left( #1 \middle| #2 \right)}
\def\tilde{\widetilde}

\newcommand\flr[1]{\lfloor #1 \rfloor}
\newcommand{\vide}{\textup{\O}}

\newtheorem{thm}{Theorem}[section]
\newtheorem{thmm}{Theorem}
\newtheorem{corr}[thmm]{Corollary}

\newtheorem{defi}[thm]{Definition}
\newtheorem{lem}[thm]{Lemma}
\newtheorem{prop}[thm]{Proposition}

\newtheorem*{nota}{Notation}

\theoremstyle{remark}
\newtheorem{remm}[thmm]{Remark}
\newtheorem{rem}[thm]{Remark}

\title{Freiheitssatz and phase transition for the density model of random groups}
\author{\textsc{Tsung-Hsuan Tsai}}
\affil{\small{Institut de Recherche Mathématique Avancée, 7 Rue René Descartes, 67000 Strasbourg, France}\\
tsung-hsuan.tsai@math.unistra.fr\\
https://orcid.org/0000-0002-5780-9260}
\affil{Key words: random group, Freiheitssatz, Gromov density model, van Kampen diagram}
\date{}

\begin{document}
\maketitle
\begin{abstract}
    Magnus' Freiheitssatz \cite{Mag30} states that if a group is defined by a presentation with $m$ generators and a single cyclically reduced relator, and this relator contains the last generating letter, then the first $m-1$ letters freely generate a free subgroup.
    
    We study an analogue of this theorem in the Gromov density model of random groups \cite{Gro93}, showing a phase transition phenomenon at density $d_r = \min\{\frac{1}{2}, 1-\log_{2m-1}(2r-1)\}$ with $1\leq r\leq m-1$: we prove that for a random group with $m$ generators at density $d$, if $d < d_r$ then the first $r$ letters freely generate a free subgroup; whereas if $d > d_r$ then the first $r$ letters generate the whole group.
    
    This result partially answers a general problem proposed by Gromov in 2003 \cite{Gro03}: existence/nonexistence of non-free subgroups in a random group.
\end{abstract}
\tableofcontents

\section{Introduction}

The \textit{Freiheitssatz} (\textit{freedom theorem} in German) is a fundamental theorem in combinatorial group theory. It was proposed by M. Dehn and proved by W. Magnus in his doctoral thesis \cite{Mag30} in 1930 (see \cite{LS77} II.5). The theorem states that for a group presentation $G = \langle x_1,\dots,x_m |r\rangle$ where the single relator $r$ is a cyclically reduced word, if $x_m$ appears in $r$, then $x_1,\dots, x_{m-1}$ freely generate a free subgroup of $G$.

\textit{Random groups} are groups obtained by a probabilistic construction. Its first mentions, in terms of ``generic property'' for finitely presented groups, appear in the works of V. S. Guba \cite{Guba86} and M. Gromov \cite{Gro87} \S0.2 in the late 1980s. The simplest model of random groups is the \textit{few relator model} (\cite{Oll05} Definition 1). A few relator random group is defined by a group presentation $G_\ell = \langle x_1,\dots, x_m|r_1,\dots, r_k \rangle$ where the set of generators $X = \{x_1,\dots,x_m\}$ is fixed, and the relators $r_1,\dots,r_k$ are chosen uniformly at random among all reduced words of $X^\pm$ of length at most $\ell$. The first well-known result of random groups (\cite{Gro87} \S0.2) is that \textit{asymptotically almost surely} (denoted by a.a.s., which means with probability converges to $1$ when $\ell$ goes to infinity), a few relator random group $G_\ell$ is non-elementary hyperbolic.

For detailed surveys on random groups, see (in chronological order) \cite{Ghys04} by E. Ghys, \cite{Oll05} by Y. Ollivier, \cite{KS08} by I. Kapovich and P. Schupp and \cite{BNW} by F. Bassino, C. Nicaud and P. Weil.

\subsection*{The density model of random groups}
In 1993, Gromov introduced the \textit{density model} of random groups in \cite{Gro93} 9.B. He considered a group presentation with a fixed set of $m$ generators $X = \{x_1,\dots,x_m\}$ and $\flr{(2m-1)^{d\ell}}$ randomly chosen relators, among the $2m(2m-1)^{\ell-1}$ reduced words of $X^\pm$ of length $\ell$. The parameter $d\in[0,1]$ is called the \textit{density}. Compare to the few relator model, the number of relators \textit{grows exponentially} with the length $\ell$. The main result of \cite{Gro93} 9.B is the \textit{phase transition} at density one half: if $d>\frac{1}{2}$, then a.a.s. the group is trivial; if $d < \frac{1}{2}$, then a.a.s. the group is non-elementary hyperbolic.

In a 1996 paper \cite{AO96}, G. Arzhantseva and A. Ol’shanskii proved a few relator random group version of the Freiheitssatz: a.a.s. every $(m-1)$-generated subgroup of a few relator random group $G_\ell$ is free. Arzhantseva proved several free subgroup properties subsequently for the few relator model in \cite{Arz97}, \cite{Arz98} and \cite{Arz00}. Kapovich-Schupp \cite{KS08} showed the existence of a small positive density $d(m)$ such that these results (\cite{AO96}, \cite{Arz97}, \cite{Arz98} and \cite{Arz00}) can be generalized to a random group at any density $d<d(m)$. It was showed in \cite{Tsai21} that the ``every $(m-1)$-generated subgroup is free'' property \cite{AO96} holds a.a.s. for a random group at any density $d<\frac{1}{120m^2\ln(2m)}$.

In 2003, Gromov defined the general notion of random groups in \cite{Gro03} and proposed in Section 1.9 the following general problem: determining asymptotic invariants and phase transition phenomena for random groups. Since then, several variants of the phase transition phenomena have been discovered. For instance, A. \.Zuk \cite{Zuk03} showed the freeness-property $(T)$ phase transition for random triangular groups at density $1/3$ (see also \cite{ALS15} by Antoniuk-Łuczak-Świ\c atkowski). Y. Ollivier proved in 2004 \cite{Oll04} the hyperbolicity-triviality transition for hyperbolic random groups, and in 2007 \cite{Oll07} the phase transition at density $1/5$ for Dehn's algorithm. In 2015 \cite{CW15}, D. Calegary and A. Walker showed that a random group at density $d<1/2$ contains surface subgroups.

As we shall see, the main result of this paper is to highlight a new phase transition phenomenon, giving an analogue of the Freiheitssatz in the density model of random groups. In particular, it partially answers Gromov's problem \cite{Gro03} 1.9 (iv): existence/nonexistence of non-free subgroups.

\subsection*{Main results}

We say that a finite group presentation $G = \langle X|R\rangle$ satisfies the \textit{Magnus Freiheitssatz property} if every subset of $X$ of cardinality $|X|-1$ freely generates a free subgroup of $G$. In particular, by Arzhantseva-Ol'shanskii's result \cite{AO96}, a few-relator random group $G_\ell$ has this property a.a.s. We study the Magnus Freiheitssatz property in the density model of random groups.

Fix a set of $m\geq 2$ elements $X = \{x_1,\dots,x_m\}$ as generators of group presentations. Denote $B_\ell$ as the set of cyclically reduced words of $X^\pm = \{x_1^\pm,\dots,x_m^\pm\}$ of length at most $\ell$. A \textit{sequence of random groups} $(G_\ell(m,d))_{\ell\in\mathbb{N}}$ with $m\geq2$ generators at density $d\in [0,1]$ is defined by group presentations $G_\ell(m,d) := \langle X|R_\ell \rangle$ where $R_\ell$ is a permutation invariant random subset of $B_\ell$ with density $d$. For example, it can be a uniform distribution on subsets of $B_\ell$ of cardinality $\flr{|B_\ell|^d}$, or a Bernoulli sampling on $B_\ell$ of parameter $|B_\ell|^{d-1}$. We are interested in the asymptotic behavior of $G_\ell(m,d)$ when $\ell\to\infty$. See Section \ref{section random subsets} or \cite{Tsai21} for detailed definitions of random groups.

An \textit{$X$-labeled graph} is a combinatorial graphs labeled by the generators $x_1,\dots,x_m$ and their inverses. The words read on the loops starting at a given vertex of the graph form a subgroup of the free group generated by $X$. Details of $X$-labeled graphs are provided in Subsection \ref{subsection Stallings graphs}.\\

The main result of this paper is a phase transition stated as follows.

\begin{thmm}[Theorem \ref{Freiheitssatz}]\label{IntroFreiheitssatz} Let $m, r$ be integers with $m\geq 2$ and $1\leq r\leq m-1$. Let $(G_\ell(m,d))$ be a sequence of random groups with $m$ generators at density $d\in[0,1]$. There is a phase transition at density
\[d_r = \min\left\{\frac{1}{2}, 1-\log_{2m-1}(2r-1)\right\}.\]
\begin{enumerate}
    \item If $d>d_r$, then a.a.s. $x_1,\dots,x_r$ generate the whole group $G_\ell(m,d)$.
    
    \item If $d<d_r$, then a.s.s. every subgroup of $G_\ell(m,d)$ generated by a reduced $X$-labeled graph $\Gamma$ with $b_1(\Gamma)\leq r$ and $|\Gamma|\leq \frac{d_r-d}{5}\ell$ is a free group of rank $r$.
    
    In particular, when $\Gamma$ is the wedge of $r$ cycles of length $1$ labeled by $x_1,\dots,x_r$ respectively, we have a.a.s. $x_1,\dots, x_r$ freely generate a free subgroup of $G_\ell(m,d)$.
\end{enumerate}
\end{thmm}

When $r=1$, the consequence of the second assertion is actually a special case of a known result: for any $m\geq 2$, if $0\leq d<1/2$, then a.a.s. the random group $G_\ell(m,d)$ is torsion free. See Remark \ref{case r=1} and \cite{Oll05} V.d.\\

Let us focus on the particular case of the second assertion. By symmetry, the set $\{x_1,\dots,x_r\}$ can be replaced by any subset $X_r$ of $X$ of cardinality $r$. In particular, if $0\leq d<d_{m-1}$, then the group presentation $G_\ell(m,d) = \langle X|R_\ell \rangle$ has the Magnus Freiheitssatz property. More precisely for the first assertion, we prove that if $d>d_r$ then a.a.s. any generator $x_i$ equals to a reduced word of $X_r^\pm$ of length $\ell-1$ in $G_\ell(m,d)$. Therefore, any relator $r_i\in R_\ell$ can be replaced by a reduced word $r'_i$ of $X_r^\pm$ of length at most $\ell(\ell-1)$. Construct $R_\ell'$ by replacing every word of $R_\ell$, we have the following result.

\begin{corr}\label{cor of 1} Let $m \geq 2$ be an integer. Let $d_r$ be the number given in \textbf{Theorem \ref{IntroFreiheitssatz}} for $1\leq r\leq m-1$. Let $(G_\ell(m,d))$ be a sequence of random groups with $m$ generators at density $d\in[0,1]$. 

For any integer $r$ with $2\leq r\leq m-1$, if the density $d$ satisfies $d_r<d<d_{r-1}$, then a.a.s. the random group $G_\ell(m,d) = \langle X|R_\ell \rangle$ admits a presentation with $r$ generators $\langle X_r | R_\ell'\rangle$ satisfying the Magnus Freiheitssatz property.
\end{corr}

\begin{remm} We emphasize that $R_\ell'$ contains relators of lengths varying from $\ell$ to $\ell^2$. Such a presentation can not be studied using known methods in geometric or combinatorial group theory. Nevertheless, it gives us new examples of groups having the Magnus Freiheitssatz property.
\end{remm}\quad

Let $r = r(m,d)$ be the maximal number such that a.a.s. $x_1,\dots,x_r$ freely generate a free subgroup of $G_\ell(m,d)$. By the phase transition at density $\frac{1}{2}$ \cite{Gro93}, if $d>\frac{1}{2}$, then $r(m,d)=0$. If $d\leq\frac{1}{2}$, by \textbf{Theorem \ref{IntroFreiheitssatz}},
\[\frac{{(2m-1)^{1-d}-1}}{2}\leq r(m,d)\leq \frac{{(2m-1)^{1-d}+1}}{2}.\]

As shown in \textbf{Fig. \ref{chart}}, because $r(m,d)$ is an integer, there is only one choice when $d$ is not $1/2$ or one of the $d_r$. Note that the value of $r(m,d)$ is not clear when $d\in\{d_1,\dots,d_{m-1},1/2\}$.

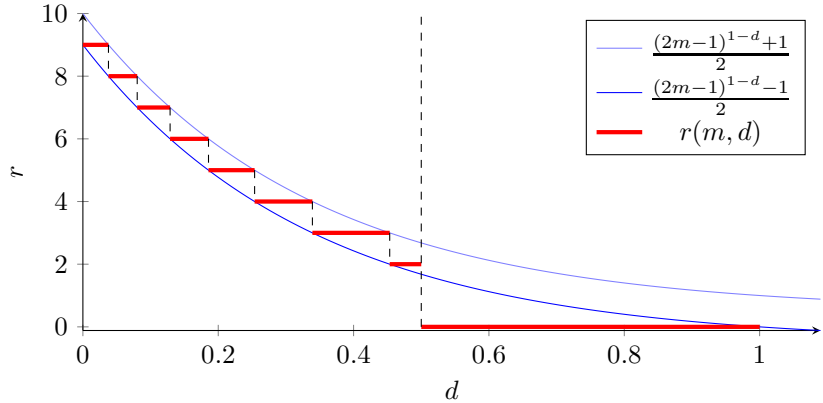
\begin{figure}[h]
\centering    
\begin{tikzpicture}[scale = 1]
    \begin{axis}[height=0.2\textheight,width=0.8\textwidth,scale only axis, axis lines = left, xlabel = {$d$}, ylabel = {$r$}]
    \addplot [
        domain = 0:1.09, 
        samples = 100, 
        color = blue!50,]
        {(19^(1-x)+1)/2};
    \addlegendentry{$\frac{{(2m-1)^{1-d}+1}}{2}$}
    \addplot [
        domain=0:1.09, 
        samples=100, 
        color=blue,]
        {(19^(1-x)-1)/2};
    \addlegendentry{$\frac{{(2m-1)^{1-d}-1}}{2}$}

    \addplot [
        domain = 1-ln(5)/ln(19):1/2, 
        samples = 5,
        color = red,
        ultra thick]
        {2};
    \addlegendentry{$r(m,d)$}

    \addplot [
        domain = 1/2:1, 
        samples = 2,
        color = red,
        ultra thick]
        {0};
    \addplot[dashed] coordinates {(1/2,0) (1/2,10)};
    
    \foreach \r in{9,...,3} {
    \addplot
        [domain = 1-ln(2*\r+1)/ln(19):1-ln(2*\r-1)/ln(19), 
        samples = 2,
        color = red,
        ultra thick]
        {\r};
    \addplot[dashed] coordinates {(1-ln(2*\r-1)/ln(19),\r)     (1-ln(2*\r-1)/ln(19),\r-1)};
    }
    \end{axis}
\end{tikzpicture}
\caption{$r(m,d)$ with $m = 10$}
\label{chart}
\end{figure}

Our main theorem (Theorem \ref{Freiheitssatz}) is a generalized version of \textbf{Theorem \ref{IntroFreiheitssatz}}. In the second assertion, we can replace the set $\{x_1,\dots,x_r\}$ by any set of $r$ words of $X^\pm$ of lengths at most $\frac{d_r-d}{5r}\ell$.

\subsection*{Outline of the paper}

In Section 2, we first recall some essential tools in combinatorial group theory (Stallings graphs \cite{Sta83} and van Kampen diagrams \cite{vK33}). We introduce \textit{distortion van Kampen diagrams} to study the distortion of subgroups of a finitely presented group.

In order to give a concrete construction of random groups with density, we discuss probabilistic models of random subsets in Section 3 and recall the \textit{intersection formula} by M. Gromov in \cite{Gro93}. Technical details are treated in \cite{Tsai21}.

Section 4 is dedicated to abstract van Kampen diagrams defined by Y. Ollivier in \cite{Oll04}. We apply his idea to distortion diagrams and define \textit{abstract distortion diagrams}. The main technical lemma for our main theorem (Theorem \ref{Freiheitssatz}) is to estimate the \textit{number of fillings} of a given abstract distortion diagram (Lemma \ref{number of fillings}).

In the last section, we state a local result on distortion van Kampen diagrams (Lemma \ref{local diagrams}) and prove the main theorem by this lemma. The last subsection is then devoted to the proof of Lemma \ref{local diagrams}.\\

\textbf{Acknowledgements.} I would like to thank my thesis advisor, Thomas Delzant, for his patience and guidance, and for many interesting and helpful discussions on the subject. I am also very grateful to the anonymous referee for his/her thorough review of the manuscript and greatly appreciate the comments and suggestions.

\section{Preliminaries on group theory}

In this section, we fix a finite group presentation $G = \langle X|R\rangle$ where $X$ is the set of generators and $R$ is the set of relators. A word $u$ in the alphabet $X^\pm$ is called \textit{reduced} if it has no sub-words of type $xx\moinsun$ or $x\moinsun x$ for any $x\in X$. If $u$ and $v$ are words that represent the same element in $G$, we denote $u =_G v$.

We consider oriented combinatorial graphs and 2-complexes as defined in Chapter III.2. of Lyndon and Schupp \cite{LS77}. Hence a \textit{graph} is a pair $\Gamma = (V,E)$ where $V$ is the set of \textit{vertices} (also called points) and $E$ is the set of (oriented) \textit{edges}. Every edge $e\in E$ has a starting point $\alpha(e)\in V$, an ending point $\omega(e)\in V$ and an inverse edge $e\moinsun\in E$, satisfying $\alpha(e\moinsun)=\omega(e)$, $\omega(e\moinsun) = \alpha(e)$ and $(e\moinsun)\moinsun = e$. The vertices $\alpha(e)$ and $\omega(e)$ are called the endpoints of $e$. An \textit{undirected edge} is a pair of inverse edges $\{e,e\moinsun\}$. 

A \textit{path} on a graph $\Gamma$ is a non-empty finite sequence of edges $p = e_1\dots e_k$ such that $\omega(e_i) = \alpha(e_{i+1})$ for $i\in\{1,\dots k-1\}$. The starting point and the ending point of the path $p$ are defined by $\alpha(p) = \alpha(e_1)$ and $\omega(p) = \omega(e_k)$. The inverse of $p$ is the path $p\moinsun = e_k\moinsun\dots e_1\moinsun$. A path is called \textit{reduced} if there is no subsequence of the form $ee\moinsun$. A \textit{loop} is a path whose starting point and ending point coincide. In this case $\alpha(p) = \omega(p)$ is called the starting point of the loop. A loop $p = e_1\dots e_k$ is \textit{cyclically reduced} if it is a reduced path with $e_k\neq e_1\moinsun$.\\

An \textit{arc} of a graph $\Gamma$ is a reduced path passing only by vertices of degree $2$, except possibly for its endpoints. A \textit{maximal arc} is an arc that can not be extended to another arc. Note that the endpoints of a maximal arc are not of degree $2$.

Denote by $b_1(\Gamma)$ the first Betti number of a graph $\Gamma$, which is the rank of its fundamental free group. The following two elementary facts for finite connected graphs can be deduced by Euler's characteristic.

\begin{lem}\label{number of vertices and maximal arcs} Let $r\geq 1$ be an integer. Let $\Gamma$ be a finite connected graph with $b_1(\Gamma) = r$ and with no vertices of degree $1$.
\begin{enumerate}
    \item The number of vertices of degree at least $3$ is bounded by $2(r-1)$.
    \item The number of maximal arcs of $\Gamma$ is bounded by $3(r-1)$.\qed
\end{enumerate}
\end{lem}

\begin{lem}\label{number of topological types} Let $r\geq 1$ be an integer. The number of topological types of finite connected graphs $\Gamma$ with $b_1(\Gamma)\leq r$ with no vertices of degree $1$ is bounded by $(2r)^{6r}$.
\end{lem}
\begin{proof}
If $r=1$ then the only topological type is a simple cycle. If $r\geq 2$, we may draw $a\leq 3(r-1)$ arcs on a set of $v\leq 2(r-1)$ vertices. There are at most $(v^2)^a \leq (2r)^{6r}$ ways. 
\end{proof}

\subsection{Stallings graphs (\texorpdfstring{$X$}{X}-labeled graphs generating subgroups)}\label{subsection Stallings graphs}

The idea of representing subgroups of a free group by graph immersions is introduced by J. Stallings in \cite{Sta83}. It is then interpreted as graphs labeled by the generators of the free group by S. Margolis and J. Meakin in \cite{MM93}. The strategy is applied by G. Arzhantseva and A. Ol'shanskii in \cite{AO96} to study subgroups of a few relator random group, then in Arzhantseva's subsequent articles \cite{Arz97}, \cite{Arz98} and \cite{Arz00}.

In this article, we follow the methods in \cite{AO96} to study subgroups of a random group in the Gromov density model.\\

A \textit{$X$-labeled graph} (Stallings graph) is a graph $\Gamma = (V,E)$ with a labelling function on edges by generators $\varphi : E\to X^\pm$, satisfying $\varphi(e\moinsun) = \varphi(e)\moinsun$. We denote briefly $\Gamma = (V,E,\varphi)$. The labeling function $\varphi$ extends naturally on the paths of $\Gamma$. If $p = e_1\dots e_k$ is a path of $\Gamma$, then the word $\varphi(p) = \varphi(e_1)\dots\varphi(e_k)$ is called the \textit{labeling word} of $p$. We say that a word $u$ is \textit{readable} on a $X$-labeled graph $\Gamma$ if there exists a path $p$ of $\Gamma$ whose labeling word is $u$.

Let $\Gamma = (V,E,\varphi)$ be a finite connected $X$-labeled graph. Labeling words of the loops starting at a vertex $o\in V$ form a subgroup $H$ of $G=\langle X|R\rangle$, which is the image of the fundamental group $\pi_1(\Gamma,o)$ by the group homomorphism induced by $\varphi$.

Since we are interested in the freeness of a subgroup, conjugacy preserving operations do not matter. Observe that the conjugacy class of the obtained subgroup is unchanged by the following three \textit{reduction operations} on the graph with a based vertex $(\Gamma,o)$ :
\begin{itemize}
    \item Change the base vertex $o\in V$.
    \item Fold a pair of edges with the same label and the same starting point.
    \item Eliminate a vertex of degree $1$ together with its only adjacent edge.
\end{itemize}

\begin{defi}[c.f. \cite{AO96} \S1]\label{labeled graph} A $X$-labeled graph is called \textbf{reduced} if it has no pair of edges with the same label and the same starting point, and, it has no vertices of degree $1$.
\end{defi}

\begin{defi} If a subgroup $H$ is a \textbf{conjugate} of $\varphi(\pi_1(\Gamma,o))$ in $G = \langle X|R\rangle$ for some vertex $o$ of $\Gamma$, we say that $H$ is \textit{generated} by the $X$-labeled graph $\Gamma$.
\end{defi}

Conversely, any finitely generated subgroup $H$ of rank $r$ can be generated by a \textit{reduced} $X$-labeled graph of first Betti number $r$. One can choose a system of generators $h_1,\dots,h_r$ of $H$, label them on the wedge of $r$ simple cycles of lengths $|h_1|,\dots,|h_r|$, and apply the three reduction operations.

\subsection{Van Kampen diagrams}
We consider van Kampen diagrams defined by Lyndon and Schupp in \cite{LS77} Chapter III.9. A \textit{2-complex} is a triplet $W = (V,E,F)$, where $(V,E)$ is a graph and $F$ is the set of (oriented) faces. Every face $f\in F$ has a boundary $\partial f$, which is a cyclically reduced loop of $(V,E)$, and an inverse face $f\moinsun\in F$ satisfying $\partial(f\moinsun) = (\partial f)\moinsun$ and $(f\moinsun)\moinsun = f$. An \textit{undirected face} is a pair of inverse faces $\{f,f\moinsun\}$. The size $|W|$ is the number of undirected faces.

Note that our definition is slightly more precise than \cite{LS77}: Every face $f\in F$ has a starting point and an orientation given by $\partial f$. If $\partial f = e_1\dots e_k$, we say that $e_i$ is \textit{attached} to $f$ and is the $i$-th boundary edge of $f$ for $1\leq i\leq k$. In this case, we say that $\{e_i,e_i\moinsun\}$ is attached to $\{f,f\moinsun\}$. An edge is called \textit{isolated} if it is not attached to any face.\\

A \textit{van Kampen diagram} (with respect to $G = \langle X | R\rangle$) is a finite, planar (embedded in $\mathbb{R}^2$) and simply connected 2-complex $D=(V,E,F)$ with two compatible labeling functions, on edges by generators $\varphi_1 : E\to X^\pm$ and on faces by relators $\varphi_2 : F\to R^\pm$. Compatible means that $(V,E,\varphi_1)$ is a $X$-labeled graph, $\varphi_2(f\moinsun) = \varphi_2(f)\moinsun$ and $\varphi_1(\partial f) = \varphi_2(f)$. Note that if a diagram $D$ has no isolated edges (for example, a disk), then $\varphi_1$ is determined by $\varphi_2$. We denote briefly $D = (V,E,F,\varphi_1,\varphi_2)$.

According to \cite{CH82} p.159, a van Kampen diagram is either a disk or a concatenation of disks and segments. The boundary $\partial D$ is the boundary of $\mathbb{R}^2\backslash D$, which is a sub-graph of its underlying graph $(V,E)$. A \textit{boundary path} is a path on $\partial D$ defined in a natural way in \cite{LS77} p.150. A \textit{boundary word} of $D$ is then the labeling word of a boundary path, unique up to cyclic conjugations and inversions. The \textit{boundary length} of $D$ is the length of a boundary path, denoted $|\partial D|$.\\

Let $D = (V,E,F,\varphi_1,\varphi_2)$ be a van Kampen diagram. A pair of faces $f,f'\in F$ is \textit{reducible} if they have the same label and there is a common edge on their boundaries at the same position (see \textbf{Fig.} \textbf{\ref{reducible faces}}). A van Kampen diagram is called \textit{reduced} if there is no reducible pair of faces.

\begin{figure}[h]
    \centering
    \begin{tikzpicture}
        \fill[gray!20] (0,0) -- (0,1) --+ (30:1) -- (30:2);
        \fill[gray!20] (0,0) -- (-30:1) --+ (30:1) -- (30:2);
        \fill[gray!20] (0,0) -- (0,1) --+ (150:1) -- (150:2);
        \fill[gray!20] (0,0) -- (210:1) --+ (150:1) -- (150:2);
        \draw[very thick] (0,0) -- (0,1) --+ (30:1) -- (30:2);
        \draw[very thick] (0,0) -- (-30:1) --+ (30:1) -- (30:2);
        \draw[very thick] (0,1) --+ (150:1) -- (150:2);
        \draw[very thick] (0,0) -- (210:1) --+ (150:1) -- (150:2);
        \draw [thick, ->] (30:2) arc (0:-30:1.5);
        \draw [thick, ->] (150:2) arc (180:210:1.5);
        \fill (30:2) circle (0.07);
        \fill (150:2) circle (0.07);
        \node at (30:1) {$r$};
        \node at (150:1) {$r$};
    \end{tikzpicture}
    \caption{a reducible pair of faces}
    \label{reducible faces}
\end{figure}
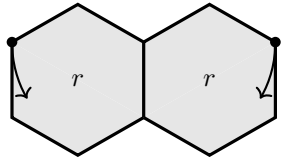    

In 1933, E. van Kampen showed in \cite{vK33} that a word $u$ of $X^\pm$ is trivial in a finitely presented group $G=\langle X|R\rangle$ if and only if it is a boundary word of a van Kampen diagram of $G$. In \cite{Ols91} \S11.6, A. Ol'shanskii improved this result to \textit{reduced} diagrams.

\begin{lem}[Van Kampen's lemma, Ol'shanskii's version]\label{van Kampen's lemma} A word $w$ of $X^\pm$ is trivial in $G = \langle X|R \rangle$ if and only if it is a boundary word of a \textbf{reduced} van Kampen diagram.
\end{lem}

\subsection{Distortion van Kampen diagrams}
Let $G = \langle X|R \rangle$ be a group presentation. For any word $u$ of $X^\pm$, we denote $|u|$ its word length and $\|u\|_G$ the distance between the endpoints of its image in the Cayley graph $\Cay(G,X)$. 

Let $\Gamma$ be a finite, connected and reduced $X$-labeled graph. Its universal covering $\tilde{\Gamma}$ is an infinite, connected and reduced labeled tree, with a natural label-preserving graph morphism $\tilde{\Gamma}\to \Cay(G,X)$. If the map $\tilde{\Gamma}\to \Cay(G,X)$ is a $\lambda$-bi-Lipschitz embedding for some $\lambda\geq 1$, then every reduced word $u$ readable on $\Gamma$ satisfies $\|u\|_G \geq \frac{1}{\lambda}|u| >0$, including those that form loops on $\Gamma$. Hence the freeness of a subgroup generated by $\Gamma$.\\

To solve this word problem, we introduce \textit{distortion van Kampen diagrams}.

\begin{defi}[Distortion diagram]\label{distortion diagram} A distortion van Kampen diagram of $(G,\Gamma)$ is a pair $(D,p)$ where $D$ is a van Kampen diagram of $G$ and $p$ is a cyclic sub-path of $\partial D$ whose labeling word is readable on $\Gamma$. (See \textbf{Fig. \ref{fig distorsion diagram}}.)
\end{defi}

\begin{figure}[h]
    \centering
    \begin{tikzpicture}
            \fill[gray!20] (0,0) circle (1.5);
            \draw[thick] (0,0) circle (1.5);
            \draw[thick] (-1.5,0) arc (-90:0:1.5);
            \draw[thick] (1.5,0) arc (-90:-145:2);
            \draw[thick] (250:1.5) arc (130:99:4);
            \draw[thick] (0,-1.05) -- (-1.1,0.05);
            \draw[thick] (0.8,0.15) -- (0.8,-0.7);
            \draw[ultra thick, blue] (1.5,0) arc (0:300:1.5);
            \node at (1.6,0.8) {$p$};
            \node at (0,0) {$D$};       
            \node at (5.8,1) {$\Gamma$};
            \draw[thick, ->] (2,1) arc (110:60:2);
            \begin{scope}[xshift = 150, yshift = 10, xscale=0.8, yscale=0.6, rotate=-90]
            \draw [thick] (0,0) arc(-225:45:1) -- (0,{sqrt(2)}) arc (225:-45:1) -- cycle;
            \draw [thick] (0,0) arc(210:150:1.4);
            \end{scope}
        \end{tikzpicture}
    \caption{The path $p$ on the boundary $\partial D$ is readable on a graph $\Gamma$}
    \label{fig distorsion diagram}
\end{figure}
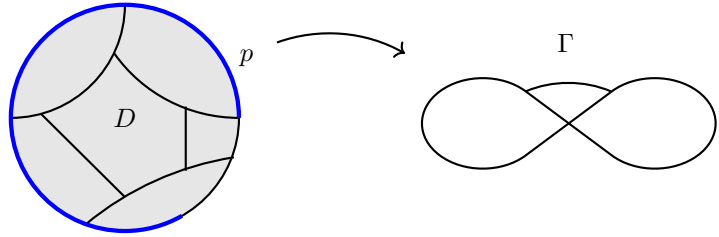

\begin{lem}\label{fff} Let $\lambda \geq 1$. If every \textbf{disk-like} and \textbf{reduced} distortion van Kampen diagram $(D,p)$ of $(G,\Gamma)$ satisfies 
\[|p|\leq \frac{\lambda}{1+\lambda}|\partial D|,\tag{$\star$}\]
then the map $\tilde{\Gamma}\to \Cay(G,X)$ is a $\lambda$-bi-Lipschitz embedding.

In particular, any subgroup generated by $\Gamma$ is free.  
\end{lem}
\begin{proof}
Let $u$ be a reduced word that is readable on $\Gamma$. Let $v$ be one of the shortest word (whose image is a geodesic in $G$) such that $uv =_G 1$. We shall check that $|u|\leq \lambda |v|$. 

By van Kampen's lemma (Lemma \ref{van Kampen's lemma}), there exists a reduced van Kampen diagram $D$ whose boundary word is $uv$. If $D$ is disk-like, then by the hypothesis $(\star)$ we have $|u|\leq \frac{\lambda}{1+\lambda}(|u|+|v|)$, which gives $|u|\leq \lambda|v|$.

Otherwise, we decompose $D$ into disks and segments $D_1,\dots,D_k$ (as in \cite{CH82} p.159). The path of $v$ does not intersect itself because it is a geodesic in $G$. The path of $u$ on $D$ does not intersect itself. If it did, as $u$ is reduced, there would be a disk-like sub-diagram whose boundary word is readable on $\Gamma$, which is impossible because of $(\star)$. 

Hence, for any $1\leq i\leq k$, there are exactly two vertices on $\partial D_i$ separating $u$ and $v$, which are the only possible vertices of degree not equal to $2$. The boundary word of $D_i$ is written as $u_iv_i$ where $u_i$ is a subword of $u$ and $v_i$ is a subword of $v$. If $D_i$ is a segment, then it is read once by $u$ and once by $v$ with opposite directions, so $|u_i|=|v_i|\leq \lambda|v_i|$. If $D_i$ is a disk, then $|u_i|\leq \lambda|v_i|$ by $(\star)$. We conclude that 
\[|u| = \sum_{i=1}^{k}|u_i| \leq \sum_{i=1}^{k}\lambda|v_i| = \lambda|v|.\]
\end{proof}

\subsection{Hyperbolic groups}
In this subsection, we recall several facts of hyperbolic groups defined by M. Gromov in \cite{Gro87}. Let $G = \langle X|R \rangle$ be a finite group presentation. The Cayley graph $\Cay(G,X)$ with the usual length metric is $\delta$-hyperbolic if each side of any geodesic triangle is $\delta$-close to the two other sides (\cite{CDP90} Chapter 1). In this case, $G$ is called a \textit{hyperbolic group}.

We start by a criterion of hyperbolicity in \cite{Gro87} Chapter 2.3. See also \cite{Sho91} by H. Short and \cite{CDP90} Chapter 6. For a precise estimation of hyperbolicity constants, see \cite{Ch94} Lemma 3.11 by C. Champetier.

\begin{thm}[Isoperimetric inequality] \label{hyperbolicity criterion} Let $\ell$ be the longest relator length in $R$. The group $G = \langle X|R\rangle$ is hyperbolic if and only if there exists a real number $\beta > 0$ such that every reduced van Kampen diagram $D$ satisfies the following isoperimetric inequality :
\[|\partial D| \geq \beta\ell|D|.\]
In this case, the Cayley graph $\Cay(G,X)$ is $\delta$-hyperbolic with
\[\delta = \frac{4\ell}{\beta}.\]
\end{thm}\quad

The local-global principle of hyperbolicity is due to M. Gromov in \cite{Gro87}. For other proofs, see \cite{Bow91} Chapter 8 by B. H. Bowditch or \cite{Pap96} by P. Papasoglu. We state here a sharpened version by Y. Ollivier in \cite{Oll07} Proposition 8.

\begin{thm}\label{local-global-hyperbolicity}(Local-global principal of hyperbolicity) For any $\alpha > 0$ and $\varepsilon>0$, there exists an integer $K = K(\alpha,\varepsilon)$ such that, if every reduced disk-like diagram $D$ with $|D|\leq K$ satisfies
\[|\partial D|\geq \alpha\ell|D|,\]
then every reduced diagram $D$ satisfies
\[|\partial D|\geq (\alpha - \varepsilon)\ell|D|.\]
\end{thm}\quad

Recall that a path $p$ in $\Cay(X,R)$ is a $\lambda$-quasi-geodesic if every sub-path $u$ of $p$ satisfies $|u|\leq \lambda\|u\|_G.$ It is a $L$-local $\lambda$-quasi geodesic if such an inequality is satisfied by every sub-path of length at most $L$. Here is the local-global principle for quasi-geodesics in hyperbolic groups, stated by Gromov in \cite{Gro87} 7.2.A and 7.2.B. See \cite{CDP90} Chapter 3 for a proof.
\begin{thm}\label{local-global-quasi-geodesics} Let $G = \langle X | R \rangle $ be a group presentation such that $\Cay(G,X)$ is $\delta$-hyperbolic. Let $\lambda \geq 1$.
\begin{enumerate}
    \item Every $\lambda$-quasi-geodesic is $100\delta(1+\log\lambda)$ close to any geodesic joining its endpoints.
    \item Every $1000\lambda\delta$-local $\lambda$-quasi-geodesic is a (global) $2\lambda$-quasi-geodesic.
\end{enumerate}
\end{thm}

\section{Random subsets and random groups}\label{section random subsets}
In this section, we recall the definition of \textit{random groups with density} by M. Gromov in \cite{Gro93}. Proofs of Proposition \ref{Bernoulli is densable}, Proposition \ref{k elements in a random subset}, Theorem \ref{random random intersection} and Theorem \ref{random fixed intersection} are in \cite{Tsai21}.
\subsection{Densable sequences of random subsets}
A \textit{random subset} $A$ of a finite set $E$ is a $\mathcal{P}(E)$-valued random variable, where $\mathcal{P}(E)$ is the set of subsets of $E$. We say that $A$ is \textit{permutation invariant} if $\Pr(A=a) = \Pr(A=\sigma(a))$ for any permutation $\sigma$ of $E$ and any subset $a$ of $E$.

In this subsection, we consider a sequence of finite sets $\bs E = (E_\ell)_{\ell\in\mathbb{N}}$ with $|E_\ell|\tendvers\infty$. Let $(Q_\ell)$ be a sequence of events. We say that the event $Q_\ell$ holds \textit{asymptotically almost surely} if $\Pr(Q_\ell)\tendvers 1$. We denote briefly a.a.s. $Q_\ell$. Note that the intersection of a finite number of events that hold a.a.s. is an event that holds a.a.s. In addition, we have the following proposition.

\begin{prop}\label{a.a.s. under condition} Let $\bs Q=(Q_\ell)$, $\bs R=(R_\ell)$ be sequences of events. If a.a.s. $Q_\ell$ and a.a.s. ``$R_\ell$ under the condition $Q_\ell$'', then a.a.s. $R_\ell$.
\end{prop}
\begin{proof}
Denote by $\overline{Q_\ell}$ the complement of $Q_\ell$. By the two hypotheses, $\Pr(Q_\ell)\to 1$ and $\PrCond{R_\ell}{Q_\ell}\to 1$. Either $\overline{Q_\ell}$ is empty and $\Pr(R_\ell) = \PrCond{R_\ell}{Q_\ell}\to 1$, or by the formula of total probability
\[\Pr(R_\ell)= \Pr(Q_\ell)\PrCond{R_\ell}{Q_\ell} + \Pr(\overline{Q_\ell})\PrCond{R_\ell}{\overline{Q_\ell}}\tendvers 1.\]
\end{proof}

Let $d\in \{-\infty\}\cup [0,1]$. A sequence of random subsets $\bs A = (A_\ell)$ of $\bs E = (E_\ell)$ is \textit{densable with density $d$} if the sequence of real-valued random variables 
\[\log_{|E_\ell|}(|A_\ell|)\]
converges in probability (or in distribution) to the constant $d$. We denote
\[\dens \bs A = d.\]
By definition, $\dens \bs A = d$ if and only if
\[\forall \varepsilon>0 \textup{ a.a.s. } |E_\ell|^{d-\varepsilon} \leq |A_\ell| \leq |E_\ell|^{d+\varepsilon}.\]
In particular, $\dens \bs A = -\infty$ if and only if a.a.s. $A_\ell = \vide$; $\dens \bs A = 0$ if and only if a.a.s. $A_\ell \neq \vide$ and $|A_\ell|$ is sub-exponential.\\

Here is the main example of a densable sequence of permutation invariant random subsets. The proofs of Theorem \ref{random random intersection} and Theorem \ref{random fixed intersection} are much simpler in this model (see \cite{Tsai21}).

\begin{prop}[Bernoulli density model, \cite{Tsai21} Proposition 1.12]\label{Bernoulli is densable} Let $0< d\leq 1$. Let $(A_\ell)$ be a sequence of random subsets of $(E_\ell)$ such that every element $e\in E_\ell$ is taken independently with probability $p_\ell = |E_\ell|^{d-1}$. Then $\bs A = (A_\ell)$ is a densable sequence of permutation invariant random subsets with density $d$.
\end{prop}\quad

Note that in the case $d=0$, the Bernoulli model is not densable. If $A_\ell$ is a Bernoulli sequence with density $d>0$, then for any distinct elements $e_1,\dots,e_k$ in $E_\ell$, we have $\Pr(e_1,\dots,e_k\in A_\ell) = p_\ell^k = |E_\ell|^{k(d-1)}$ by independence. This property is, in general, not true for an arbitrary densable sequence of permutation invariant random subsets. Nevertheless, it can be approached asymptotically.

\begin{prop}[Similar to \cite{Tsai21} Lemma 3.10]\label{k elements in a random subset} Let $\bs A = (A_\ell)$ be a densable sequence of permutation invariant random subsets of $\bs E = (E_\ell)$ with density $d$. Let $\varepsilon>0$. Denote $Q_\ell$ the event $|E_\ell|^{d-\varepsilon} \leq |A_\ell| \leq |E_\ell|^{d+\varepsilon}$ (we have a.a.s. $Q_\ell$ by definition). Let $e_1,\dots,e_k$ be distinct elements in $E_\ell$. For $\ell$ large enough,
\[|E_\ell|^{k(d-1-2\varepsilon)} \leq \PrCond{e_1,\dots,e_k\in A_\ell}{Q_\ell} \leq |E_\ell|^{k(d-1+2\varepsilon)}.\qed\]
\end{prop}

\subsection{The intersection formula}
We recall here the \textit{intersection formula} for random subsets. See \cite{Gro93} for the original version by M. Gromov, and \cite{Tsai21} Section 2 for a proof.\\

\begin{thm}[The intersection formula]\label{random random intersection} Let $\bs A = (A_\ell)$, $\bs B = (B_\ell)$ be independent densable sequences of permutation invariant random subsets. 
\begin{enumerate}
    \item If $\dens\bs A + \dens \bs B < 1$, then a.a.s. $A_\ell \cap B_\ell = \vide.$
    \item If $\dens \bs A + \dens \bs B > 1$, then $\bs A \cap \bs B :=(A_\ell \cap B_\ell)$ is a densable sequence of permutation invariant random subset. In addition, \[\dens(\bs A\cap \bs B) = \dens \bs A + \dens \bs B - 1.\]
    In particular, a.a.s. $A_\ell \cap B_\ell \neq \vide.$\qed
\end{enumerate}
\end{thm}\quad

A \textit{fixed} subset can be regarded as a constant random subset. The density of a sequence of fixed subsets can be defined by the same way. Note that a sequence of subsets $\bs F = (F_\ell)$ of $\bs E = (E_\ell)$ is densable with density $d$ if and only if 
\[|F_\ell| = |E_\ell|^{d+o(1)}.\]

We consider also the intersection between a sequence of random subsets and a sequence of fixed subsets. See \cite{Tsai21} Section 3 for a proof.

\begin{thm}[\cite{Tsai21} Theorem 3.7]\label{random fixed intersection} Let $\bs A = (A_\ell)$ be a densable sequence of permutation invariant random subsets of $\bs E$. Let $\bs F = (F_\ell)$ be a densable sequence of fixed subsets.
\begin{enumerate}
    \item If $\dens\bs A + \dens \bs F < 1$, then a.a.s. $A_\ell \cap F_\ell = \vide.$
    \item If $\dens \bs A + \dens \bs F > 1$, then the sequence $\bs A \cap \bs F$ is densable in $\bs E$, with density \[\dens \bs A + \dens \bs F - 1.\]
    
    In addition, $\bs A \cap \bs F$ is densable and permutation invariant in $\bs F$, with density \[\frac{\dens \bs A + \dens \bs F - 1}{\dens \bs F}.\]\qed
\end{enumerate}
\end{thm}

\subsection{The density model of random groups}
Fix an alphabet $X = \{x_1,\dots, x_m\}$ as generators of group presentations. Let $B_\ell$ be the set of cyclically reduced words on $X^\pm = \{x_1^\pm,\dots,x_m^\pm\}$ of lengths \textit{at most $\ell$}. Note that
\[|B_\ell| = (2m-1)^{\ell+o(\ell)}.\]

We consider a sequence of random groups $\bs G(m,d) = (G_\ell(m,d))$ defined by random presentations $G_\ell(m,d) := \langle X|R_\ell\rangle$ where $\bs R = (R_\ell)$ is a densable sequence of permutation invariant random subsets of $\bs B = (B_\ell)$ with density $d$. Such a sequence is called \textit{a sequence of random groups at density $d$}.

The number of relators $|R_\ell|$ is a real-valued random variable and is concentrated to $(2m-1)^{d\ell}$. More precisely, for any $\varepsilon>0$ a.a.s.
\[(2m-1)^{d\ell-\varepsilon\ell} \leq |R_\ell| \leq (2m-1)^{d\ell+\varepsilon\ell}.\]

We are interested in asymptotic behaviors of a sequence of random groups. In his book \cite{Gro93}, Gromov observed that there is a \textit{phase transition} at density $1/2$.

\begin{thm}[Phase transition at density $1/2$]\label{hyp} Let $\bs G(m,d) = (G_\ell(m,d)) = (\langle X | R_\ell\rangle)$ be a sequence of random groups at density $d$.
\begin{enumerate}
    \item If $d>1/2$, then a.a.s. $G_\ell(m,d)$ is a trivial group.
    \item If $d<1/2$, then a.a.s. $G_\ell(m,d)$ is a hyperbolic group, and the Cayley graph $\Cay(G_\ell,X)$ is $\delta$-hyperbolic with $\delta = \frac{4\ell}{1-2d}$.
    
    In addition, for any $s>0$, a.a.s. every reduced van Kampen diagram $D$ of $G_\ell(m,d)$ satisfies the isoperimetric inequality
    \[|\partial D|\geq (1-2d-s)\ell|D|.\]
\end{enumerate}
\end{thm}

The proof of our main theorem (Theorem \ref{Freiheitssatz}) is very similar to the strategy of proving this theorem: For the first point we apply the intersection formula to show that some type of relation exists; for the second assertion we convert the problem into a diagram problem and apply some local-global argument in hyperbolic groups.

For Theorem \ref{hyp}, we give here a proof for the first assertion and an idea of proof for the second assertion.
\begin{proof}[Proof of Theorem \ref{hyp}.1] Let $S_\ell$ be the set of cyclically reduced words of length exactly $\ell$. The sequence $(S_{\ell-1})$ is a fixed sequence of subsets of $\bs B = (B_\ell)$ of density $1$. By the intersection formula (Theorem \ref{random fixed intersection}), the two sequences $(x_1R_\ell\cap x_1S_{\ell-1})$ and $(R_\ell\cap x_1S_{\ell-1})$ are both sequences of random subsets of $(x_1S_{\ell-1})$ with density $d$. By the intersection formula between random subsets (Theorem \ref{random random intersection}), their intersection is a sequence of random subsets with density $(2d-1)>0$, which is a.a.s. not empty. Thus, a.a.s. there exists a word $w\in S_{\ell-1}$ such that $w\in R_\ell$ and $x_1w\in R_\ell$, so a.a.s. $x_1=1$ in $G_\ell$ by canceling $w$.

The argument works for every generator $x_i\in X$. By intersecting a finite number of a.a.s. events, a.a.s. $G_\ell$ is isomorphic to the trivial group.
\end{proof}\quad

By Theorem \ref{hyperbolicity criterion} and Theorem \ref{local-global-hyperbolicity}, to prove Theorem \ref{hyp}.2, it is sufficient to find a local isoperimetric inequality. See \cite{Oll05} for a proof by Y. Ollivier.

\begin{lem}[Local isoperimetric inequality] \label{local isoperimetric inequality} Let $s>0$. If $d<1/2$, then for $K = K\left(1-2d-\frac{s}{2},\frac{s}{2}\right)$ provided by Theorem \ref{local-global-hyperbolicity}, a.a.s. any reduced disc-like diagram $D$ of $G_\ell(m,d)$ with at most $K$ faces satisfies the isoperimetric inequality
    \[|\partial D| \geq \left(1-2d-\frac{s}{2}\right)\ell|D|.\]\qed
\end{lem}

\section{Abstract diagrams}

According to Proposition \ref{k elements in a random subset}, the probability that a van Kampen diagram appears in a given random group at density is determined by the number of relators used in this diagram. Two van Kampen diagrams having the same underlying 2-complex may not use the same number of relators, and should be treated separately. 

For example, to check that if a group satisfies the $C'(\lambda)$ small cancellation condition (see \cite{Tsai21} Theorem 4.3), we consider van Kampen diagrams whose underlying 2-complex consists of two faces $f_1,f_2$ sharing a common path of length $\lambda\min\{|\partial f_1|,|\partial f_2|\}$. We then need to consider the two types of diagrams in \textbf{Fig. \ref{two types of diagrams}}, one using two distinct relators and the other one using one relator. 

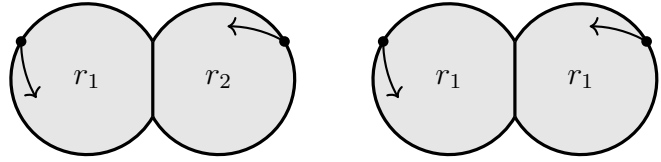
\begin{figure}[h]
    \centering
\begin{tikzpicture}
            \fill[gray!20] (30:1) circle (1) ;
            \fill[gray!20] (150:1) circle (1) ;
            \draw[very thick] (0,0) -- (0,1);
            \draw[very thick] (0,0) arc (-150:150:1);
            \draw[very thick] (0,1) arc (30:330:1);
            \node at (30:1) {\large{$r_2$}};
            \node at (150:1) {\large{$r_1$}};
            \draw [thick, ->] (30:2) arc (60:90:1.5);
            \draw [thick, ->] (150:2) arc (180:210:1.5);
            \fill (30:2) circle (0.07);
            \fill (150:2) circle (0.07);
    \end{tikzpicture}\qquad
    \begin{tikzpicture}
            \fill[gray!20] (30:1) circle (1) ;
            \fill[gray!20] (150:1) circle (1) ;
            \draw[very thick] (0,0) -- (0,1);
            \draw[very thick] (0,0) arc (-150:150:1);
            \draw[very thick] (0,1) arc (30:330:1);
            \node at (30:1) {\large{$r_1$}};
            \node at (150:1) {\large{$r_1$}};
            \draw [thick, ->] (30:2) arc (60:90:1.5);
            \draw [thick, ->] (150:2) arc (180:210:1.5);
            \fill (30:2) circle (0.07);
            \fill (150:2) circle (0.07);
    \end{tikzpicture}
    \caption{the two types of diagrams for checking if a group satisfies $C'(\lambda)$}
    \label{two types of diagrams}
\end{figure}

When the area of the common underlying 2-complex is larger, the number of types of diagrams increases and can be very sophisticated. In 2004, Y. Ollivier introduced \textit{abstract van Kampen diagrams} (\cite{Oll04} p.10) to surround this problem.

\subsection{Abstract van Kampen diagrams}
\begin{defi}[Abstract diagram, Ollivier \cite{Oll05}] An abstract van Kampen diagram $\tilde D$ is a finite, planar and simply-connected 2-complex $(V,E,F)$ with a labeling function on faces by integer numbers $\tilde\varphi_2 : F \to \{1,1^-,2,2^-,\dots,k,k^-\}$ satisfying $\tilde\varphi_2(f\moinsun) = \tilde\varphi_2(f)^-$. We denote $\tilde D = (V,E,F,\tilde\varphi_2)$,
\end{defi}
By convention, $(i^-)^- = i$ for any $1\leq i \leq k$. The numbers $\{1,\dots,k\}$ are called \textit{abstract relators} of $\tilde D$.

Similarly to a van Kampen diagram, a pair of faces $f,f'\in F$ is \textit{reducible} if they have the same label, and they share an edge at the same position of their boundaries. An abstract diagram is called \textit{reduced} if there is no reducible pair of faces.\\

Let $D = (V,E,F,\varphi_1,\varphi_2)$ be a van Kampen diagram of a group presentation $G = \langle X|R\rangle$. Let $\{r_1,\dots,r_k\}\subset R$ be the set of relators used in $D$. Define $\tilde\varphi_2 : F\to \{1,1^-,\dots,k,k^-\}$ by $\tilde\varphi_2(f) = i$ if $\varphi_2(f) = r_i$. We obtain an abstract diagram $\tilde D = (V,E,F,\tilde\varphi_2)$ with $k$ abstract relators, called an \textit{underlying abstract diagram} of $D$.\\

An abstract diagram $\tilde D$ is  \textit{fillable} by a group presentation $G = \langle X|R\rangle$ (or by a set of relators $R$) if there exists a van Kampen diagram $D$ of $G$, called a \textit{filled diagram} of $\tilde D$, whose underlying abstract diagram is $\tilde D$. That is to say, there exists $k$ \textit{different} relators $r_1,\dots,r_k\in R$ such that the construction $\varphi_2(f) := r_{\tilde\varphi_2(f)}$ gives a diagram $D = (V,E,F,\varphi_1,\varphi_2)$ of $G$. In \textbf{Fig. \ref{filling an abstract diagram}}, the abstract diagram has two abstract relators $1,2$ and is filled by the relators $r_1,r_2$. The $k$-tuple $(r_1,\dots,r_k)$ is called a \textit{filling} of $\tilde D$. As we picked different relators, $\tilde D$ is reduced if and only if a filled diagram $D$ is reduced.

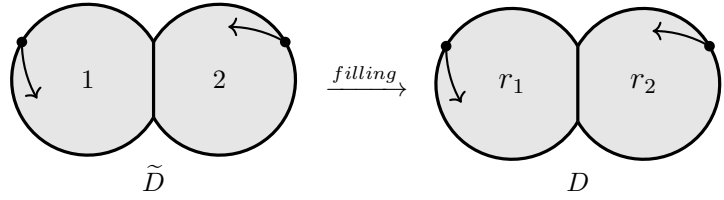
\begin{figure}[h]
\centering
    \begin{tikzpicture}
            \fill[gray!20] (30:1) circle (1) ;
            \fill[gray!20] (150:1) circle (1) ;
            \draw[very thick] (0,0) -- (0,1);
            \draw[very thick] (0,0) arc (-150:150:1);
            \draw[very thick] (0,1) arc (30:330:1);
            \node at (30:1) {$2$};
            \node at (150:1) {$1$};
            \draw [thick, ->] (30:2) arc (60:90:1.5);
            \draw [thick, ->] (150:2) arc (180:210:1.5);
            \fill (30:2) circle (0.07);
            \fill (150:2) circle (0.07);
            \node at (2.8,0.5) {$\xrightarrow{filling}$};
            \node at (0,-0.8) {$\tilde D$};
    \end{tikzpicture}
    \begin{tikzpicture}
            \fill[gray!20] (30:1) circle (1) ;
            \fill[gray!20] (150:1) circle (1) ;
            \draw[very thick] (0,0) -- (0,1);
            \draw[very thick] (0,0) arc (-150:150:1);
            \draw[very thick] (0,1) arc (30:330:1);
            \node at (30:1) {\large{$r_2$}};
            \node at (150:1) {\large{$r_1$}};
            \draw [thick, ->] (30:2) arc (60:90:1.5);
            \draw [thick, ->] (150:2) arc (180:210:1.5);
            \fill (30:2) circle (0.07);
            \fill (150:2) circle (0.07);
            \node at (0,-0.8) {$D$};
    \end{tikzpicture}
    \caption{filling an abstract diagram}\label{filling an abstract diagram}
\end{figure}

We assume that faces with the same label of $\tilde D$ have the same boundary length, otherwise $\tilde D$ would never be fillable. Denote $\ell_i$ the length of the abstract relator $i$ for $1\leq i \leq k$. Let $\ell = \max\{\ell_1, \dots, \ell_k\}$ be the maximal boundary length of faces of $\tilde D$.

\begin{nota} The pairs of integers $(i,1), \dots, (i,\ell_i)$ are called \textit{abstract letters} of $i$. 

The set of abstract letters of $\tilde D$, denoted $\tilde X$, is then a subset of $\{1,\dots,k\} \times \{1,\dots,\ell\}$, endowed with the lexicographic order.
\end{nota}

We decorate undirected edges of $\tilde D$ by abstract letters and directions. Let $f\in F$ labeled by $i$ and let $e\in E$ at the $j$-th position of $\partial f$. The edge $\{e,e\moinsun\}$ is decorated, on the side of $\{f,f\moinsun\}$, by an arrow indicating the direction of $e$ and the abstract letter $(i,j)$. This decoration on $\{e,e\moinsun\}$ is called the \textit{decoration from $f$ at the position $j$}. The number of decorations on an edge $\{e,e\moinsun\}$ is the number of its adjacent faces $\{f,f\moinsun\}$ with multiplicity (0, 1 or 2 when $\tilde D$ is planar).

For any filling $(r_1,\dots,r_k)$ of $\tilde D$, we construct the \textit{canonical function} $\phi:\tilde X\to X^\pm$ such that $r_i = \phi(i,1)\dots\phi(i,\ell_i)$ for any $1\leq i \leq k$. If an edge $\{e,e\moinsun\}$ is decorated by two abstract letters $(i,j)$, $(i',j')$, then $\phi(i',j') = \phi(i,j)$ if they have the same direction, or $\phi(i',j') = \phi(i,j)\moinsun$ if they have opposite directions. For example, in the diagram of \textbf{Fig. \ref{diagram decoration}}, there is an edge decorated by two abstract letters $(1,4)$ and $(2,3)$ with opposite directions, so we have $\varphi(1,4) = \varphi(2,3)\moinsun$.

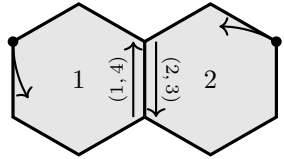
\begin{figure}[h]
    \centering
\begin{tikzpicture}
        \fill[gray!20] (0,0) -- (0,1) --+ (30:1) -- (30:2);
        \fill[gray!20] (0,0) -- (-30:1) --+ (30:1) -- (30:2);
        \fill[gray!20] (0,0) -- (0,1) --+ (150:1) -- (150:2);
        \fill[gray!20] (0,0) -- (210:1) --+ (150:1) -- (150:2);
        \draw[very thick] (0,0) -- (0,1) --+ (30:1) -- (30:2);
        \draw[very thick] (0,0) -- (-30:1) --+ (30:1) -- (30:2);
        \draw[very thick] (0,1) --+ (150:1) -- (150:2);
        \draw[very thick] (0,0) -- (210:1) --+ (150:1) -- (150:2);
        \draw [thick, ->] (30:2) arc (60:90:1.5);
        \draw [thick, ->] (150:2) arc (180:210:1.5);
        \fill (30:2) circle (0.07);
        \fill (150:2) circle (0.07);
        \draw [thick, ->] (0.15,1) -- (0.15,0);
        \draw [thick, ->] (-0.15,0) -- (-0.15,1);
        \node at (30:1) {$2$};
        \node at (150:1) {$1$};
        \node[rotate = -90] at (0.35,0.5) {\scriptsize{$(2,3)$}};
        \node[rotate = 90] at (-0.35,0.5) {\scriptsize{$(1,4)$}};
\end{tikzpicture}
    \caption{an edge decorated by two abstract letters}
    \label{diagram decoration}
\end{figure}


Note that if $\tilde D$ is reduced, then by definition an abstract letter can not be decorated twice on an edge with the same direction (\textbf{Fig. \ref{double decoration}} left-hand side). If $\tilde D$ is fillable (by the set of all relators), then an abstract letter $(i,j)$ can not be decorated twice on an undirected edge with opposite directions (\textbf{Fig. \ref{double decoration}} right-hand side), otherwise we have $\phi(i,j)=\phi(i,j)\moinsun$ in the set of generators $X$. 

\begin{figure}[h]
    \centering
    \begin{tikzpicture}
        \fill[gray!20] (0,0) -- (0,1) --+ (30:1) -- (30:2);
        \fill[gray!20] (0,0) -- (-30:1) --+ (30:1) -- (30:2);
        \fill[gray!20] (0,0) -- (0,1) --+ (150:1) -- (150:2);
        \fill[gray!20] (0,0) -- (210:1) --+ (150:1) -- (150:2);
        \draw[very thick] (0,0) -- (0,1) --+ (30:1) -- (30:2);
        \draw[very thick] (0,0) -- (-30:1) --+ (30:1) -- (30:2);
        \draw[very thick] (0,1) --+ (150:1) -- (150:2);
        \draw[very thick] (0,0) -- (210:1) --+ (150:1) -- (150:2);
        \draw [thick, ->] (30:2) arc (0:-30:1.5);
        \draw [thick, ->] (150:2) arc (180:210:1.5);
        \fill (30:2) circle (0.07);
        \fill (150:2) circle (0.07);
        \draw [thick, ->] (0.15,0) -- (0.15,1);
        \draw [thick, ->] (-0.15,0) -- (-0.15,1);
        \node at (30:1) {$1$};
        \node at (150:1) {$1$};
        \node[rotate = -90] at (0.35,0.5) {\scriptsize{$(1,4)$}};
        \node[rotate = 90] at (-0.35,0.5) {\scriptsize{$(1,4)$}};
    \end{tikzpicture}\qquad
    \begin{tikzpicture}
        \fill[gray!20] (0,0) -- (0,1) --+ (30:1) -- (30:2);
        \fill[gray!20] (0,0) -- (-30:1) --+ (30:1) -- (30:2);
        \fill[gray!20] (0,0) -- (0,1) --+ (150:1) -- (150:2);
        \fill[gray!20] (0,0) -- (210:1) --+ (150:1) -- (150:2);
        \draw[very thick] (0,0) -- (0,1) --+ (30:1) -- (30:2);
        \draw[very thick] (0,0) -- (-30:1) --+ (30:1) -- (30:2);
        \draw[very thick] (0,1) --+ (150:1) -- (150:2);
        \draw[very thick] (0,0) -- (210:1) --+ (150:1) -- (150:2);
        \draw [thick, ->] ($sqrt(3)*(1,0)$) arc (0:30:1.5);
        \draw [thick, ->] (150:2) arc (180:210:1.5);
        \fill ($sqrt(3)*(1,0)$) circle (0.07);
        \fill (150:2) circle (0.07);
        \draw [thick, ->] (0.15,1) -- (0.15,0);
        \draw [thick, ->] (-0.15,0) -- (-0.15,1);
        \node at (30:1) {$1$};
        \node at (150:1) {$1$};
        \node[rotate = -90] at (0.35,0.5) {\scriptsize{$(1,4)$}};
        \node[rotate = 90] at (-0.35,0.5) {\scriptsize{$(1,4)$}};
    \end{tikzpicture}
    \caption{an edge decorated twice by the same abstract letter}
    \label{double decoration}
\end{figure}
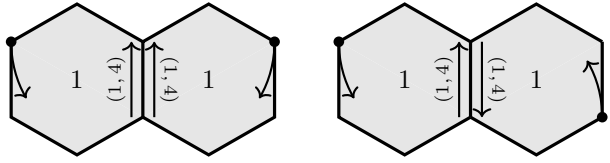

In the following, we assume that $\tilde D$ is fillable and reduced, so that the abstract letters decorated on an edge $\{e,e\moinsun\}$ are all different, and the two types of sub-diagrams in \textbf{Fig. \ref{double decoration}} can not appear in $\tilde D$. In particular, for any edge $\{e,e\moinsun\}$, there exists a unique face $\{f,f\moinsun\}$ (at a unique position) from which the decoration is (lexicographically) minimal. Whence the following two definitions.

\begin{defi}[Preferred face of an edge]\label{preferred face} Let $\{e,e\moinsun\}$ be an edge of $\tilde D$. Let $\{f,f\moinsun\}$ be the adjacent face of $\{e,e\moinsun\}$ from which the decoration is minimal. Then $\{f,f\moinsun\}$ is called the \textbf{preferred face} of $\{e,e\moinsun\}$.
\end{defi}

\begin{defi}[free-to-fill] An abstract letter $(i,j)$ of $\tilde D$ is \textbf{free-to-fill} if, for any edge $\{e,e\moinsun\}$ decorated by $(i,j)$, it is the minimal decoration on $\{e,e\moinsun\}$.
\end{defi}

Note that an abstract letter $(i,j)$ is free-to-fill if and only if every face $f$ labeled by $i$ is the preferred face of its $j$-th boundary edge. In other words, if $(i,j)$ is \textit{not free-to-fill}, then there exists an edge $\{e,e\moinsun\}$ decorated by $(i,j)$ that has another decoration $(i',j')<(i,j)$.

For example, in the abstract diagram of \textbf{Fig. \ref{example abstract diagram}}, $(1,4)$, $(2,1)$ and $(2,2)$ are not free-to-fill. The other abstract letters are free-to-fill.

\begin{figure}[h]
    \centering
\begin{tikzpicture}
        \fill[gray!20] (0,0) -- (0,1) --+ (30:1) -- (30:2);
        \fill[gray!20] (0,0) -- (-30:1) --+ (30:1) -- (30:2);
        \fill[gray!20] (0,0) -- (0,1) --+ (150:1) -- (150:2);
        \fill[gray!20] (0,0) -- (210:1) --+ (150:1) -- (150:2);
        \fill[gray!20] (0,0) -- (0,-2) --+ (30:1) -- (-30:1);        
        \fill[gray!20] (0,0) -- (0,-2) --+ (150:1) -- (-150:1);
        \draw[very thick] (0,0) -- (0,1) --+ (30:1) -- (30:2);
        \draw[very thick] (0,0) -- (-30:1) --+ (30:1) -- (30:2);
        \draw[very thick] (0,1) --+ (150:1) -- (150:2);
        \draw[very thick] (0,0) -- (210:1) --+ (150:1) -- (150:2);
        \draw[very thick] (0,-2) --+ (30:1) -- (-30:1);
        \draw[very thick] (0,-2) --+ (150:1) -- (-150:1);
        \draw [thick, ->] ($sqrt(3)*(1,0)$) arc (-60:-90:1.5);
        \draw [thick, ->] (150:2) arc (180:210:1.5);
        \draw [thick, ->] (-30:1) arc (60:90:1.5);
        \fill ($sqrt(3)*(1,0)$) circle (0.07);
        \fill (150:2) circle (0.07);
        \fill (-30:1) circle (0.07);
        \node at (30:1) {$1$};
        \node at (150:1) {$1$};
        \node at (0,-1) {$2$};
\end{tikzpicture}
    \caption{example of an abstract diagram}
    \label{example abstract diagram}
\end{figure}
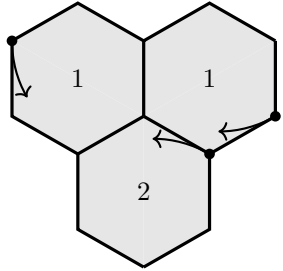

Denote $F^+ = \{f\in F\,|\, \tilde\varphi_2(f)\in\{1,\dots,k\}\}$. It gives a preferred orientation for each undirected face of $\tilde D = (V,E,F,\tilde \varphi_2)$. Let $\overline E$ be the set of undirected edges of $\tilde D$.
\begin{lem}\label{reunion of prefering edges} Let $\tilde D$ be a reduced fillable abstract diagram without isolated edges. For every face $f\in F^+$, let $\overline E_f$ be the set of edges $\{e,e\moinsun\}$ on the boundary of $\{f,f\moinsun\}$ such that $\{f,f\moinsun\}$ is the preferred face of $\{e,e\moinsun\}$. Then
    \[\overline E = \bigsqcup_{f\in F^+}\overline E_f.\]
\end{lem}
\begin{proof} For every edge $\{e,e\moinsun\}$ there exists a unique face $f\in F^+$ such that $\{e,e\moinsun\}\in\overline E_f$. Hence, the sets $\overline E_f$ with $f\in F^+$ are pairwise disjoint. Their reunion is the set of edges because every edge is adjacent to at least one face.
\end{proof}

\subsection{Abstract distortion van Kampen diagrams}
We generalize the idea of abstract diagrams to distortion van Kampen diagrams.

\begin{defi}[Abstract distortion diagram] An abstract distortion van Kampen diagram is a pair $(\tilde D,p)$ where $\tilde D$ is an abstract diagram and $p$ is a path on $\partial \tilde D$.
\end{defi}

Let $G=\langle X|R\rangle$ be a group presentation and let $\Gamma$ be a $X$-labeled graph. An abstract distortion diagram $(\tilde D,p)$ is \textit{fillable} by the pair $(G,\Gamma)$ (or by the pair $(R,\Gamma)$) if there exists a filled diagram $D$ of $\tilde D$ such that $(D,p)$ is a distortion diagram of $(G,\Gamma)$. The distortion diagram $(D,p)$ is called a \textit{filled distortion diagram} of $(\tilde D,p)$.\\

In the following, an abstract distortion diagram $(\tilde D,p)$ is reduced, fillable and without isolated edges. Recall that $\tilde X \subset \{1,\dots,k\} \times \{1,\dots,\ell\}$ is the set of abstract letters. Let $\overline p$ be the set of undirected edges given by $p$. In an abstract distortion diagram, we distinguish between two types of free-to-fill abstract letters: those that decorate an edge of $\overline p$ and those that do not.

\begin{defi} Let $(i,j)$ be an abstract letter of $(\tilde D,p)$.
 \begin{enumerate}[$(i)$]
    \item $(i,j)$ is \textbf{free-to-fill} if it is free-to-fill for the abstract diagram $\tilde D$ and it does not decorate any edge of $\overline p$.
    \item $(i,j)$ is \textbf{semi-free-to-fill} if it is free-to-fill for the abstract diagram $\tilde D$ and it decorates an edge of $\overline p$.
    \item Otherwise, $(i,j)$ is not free-to-fill.
\end{enumerate}
\end{defi}\quad

\begin{nota} Let $i$ be an abstract relator of $\tilde D$. We denote $\alpha_i$ the number of faces labeled by $i$, $\eta_i$ the number of free-to-fill abstract letters of $i$, and $\eta'_i$ the number of semi-free-to-fill abstract letters of $i$.
\end{nota}
Note that $\ell_i-\eta_i-\eta'_i$ is the number of non free-to-fill edges.
\begin{lem}\label{number of free to fill edges} Recall that $\overline E_f$ is the set of edges on the boundary of $f$ that prefers $\{f,f\moinsun\}$. Let $i$ be an abstract relator. For any face $f\in F$ with $\tilde\varphi_2(f)=i$, we have
    \[\eta'_i\leq |\overline E_f\cap \overline p| \textup{\quad and\quad} \eta_i\leq |\overline E_f|-|\overline E_f\cap \overline p|.\]
\end{lem}
\begin{proof} 
Let $\{e,e\moinsun\}$ be the edge at the $j$-th position of $\partial f$. It is decorated by $(i,j)$. If $\{f,f\moinsun\}$ is not preferred by $\{e,e\moinsun\}$, then $(i,j)$ is not free-to-fill because there is a smaller decoration on $\{e,e\moinsun\}$. 

Thus, if $\{e,e\moinsun\} \in \overline E_f\cap \overline p$ then $(i,j)$ is semi-free-to-fill, which gives the first inequality. Similarly, if $\{e,e\moinsun\} \in \overline E_f\backslash\overline p$, then $(i,j)$ is free-to-fill, so we have the second inequality.
\end{proof}\quad

\begin{lem}\label{alphai etai} Recall that $\overline{E}$ is the set of undirected edges. The following two inequalities hold.
    \[\sum_{i=1}^k \alpha_i\eta'_i \leq |\overline p|, \quad \sum_{i=1}^k \alpha_i\eta_i \leq |\overline E|-|\overline p|.\]
\end{lem}
\begin{proof} By Lemma \ref{number of free to fill edges}, for every $1\leq i\leq k$
\[\alpha_i\eta'_i\leq\sum_{f\in F,\tilde\varphi_2(f) = i}|\overline E_f\cap \overline p|.\]
Apply Lemma \ref{reunion of prefering edges},
\[\sum_{i=1}^k\alpha_i\eta'_i\leq\sum_{f\in F^+}|\overline E_f\cap \overline p|\leq |\overline p|.\]
We get the second inequality by replacing $\eta'_i$ by $\eta_i$ and $|\overline p|$ by $|\overline E \backslash \overline p|$.
\end{proof}

\subsection{The number of fillings of an abstract distortion diagram}
Recall that $X = \{x_1,\dots,x_m\}$ is a fixed set of $m\geq 2$ generators and that $B_\ell$ is the set of cyclically reduced words of $X^\pm = \{x_1^\pm,\dots,x_m^\pm\}$ of length at most $\ell$. For a graph $\Gamma$, denote by $|\Gamma|$ the number of its undirected edges and $b_1(\Gamma)$ its first Betti number, which is the rank of its fundamental free group.\\

Throughout this subsection, we fix an integer $1\leq r\leq m-1$, a $X$-labeled graph $\Gamma$ with $b_1(\Gamma)= r$, and an abstract distortion diagram $(\tilde D,p)$ with $k\geq 1$ abstract relators. Assume that $\tilde D$ is reduced, fillable and has no isolated edges, and that $\ell$ is the longest boundary length of faces of $\tilde D$.

Denote by $N_\ell(\tilde D, p,\Gamma)$ the set of fillings $(r_1,\dots,r_k)$ of $(\tilde D,p)$ by $(B_\ell,\Gamma)$. In Lemma \ref{number of fillings}, we give an upper bound of the \textit{number of fillings} $|N_\ell(\tilde D,p,\Gamma)|$.

\begin{lem}\label{number of readable words on a graph} The number of reduced words $u$ of length $L$ that is readable on $\Gamma$ is at most $2|\Gamma|(2r-1)^L$.
\end{lem}
\begin{proof}
We estimate the number of paths $p$ on $\Gamma$ whose labeling word can be reduced. Take an oriented edge of $\Gamma$ as the first edge of $p$, there are $2|\Gamma|$ choices. Every vertex is of degree at most $2r$ because $b_1(\Gamma)= r$. As $p$ is reduced, every time we take the next edge, there are at most $(2r-1)$ choices. Hence, there are at most $2|\Gamma|(2r-1)^L$ paths.
\end{proof}

\begin{rem}\label{rem of lem readable} Note that when $r=1$, a reduced $X$-labeled graph $\Gamma$ with $b_1(\Gamma)=1$ is a simple cycle (since there is no isolated edges). The number of reduced words $u$ of length $L$ that is readable on $\Gamma$ is exactly $2|\Gamma|$, as the choice of a starting point and an orientation on $\Gamma$ decides such a word.
\end{rem}\quad

A vertex of $(\tilde D,p)$ is called \textit{distinguished} if it is either of degree at least $3$, or the starting point of a face, or an endpoint of $p$. Let $i$ be an abstract letter of $(\tilde D,p)$. It can be regarded as a 2-complex (see \textbf{Fig. \ref{2-complex of i}}) with two inverse faces $\{i,i^-\}$ and $2\ell_i$ edges $(i,1),\dots,(i,\ell_i)$ with their inverses, such that $\partial i = (i,1)\dots(i,\ell_i)$.

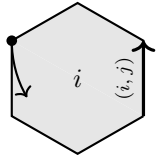
\begin{figure}[h]
    \centering
\begin{tikzpicture}
        \fill[gray!20] (0,0) -- (0,1) --+ (150:1) -- (150:2);
        \fill[gray!20] (0,0) -- (210:1) --+ (150:1) -- (150:2);
        \draw[thick] (0,0) -- (0,1) --+ (150:1) -- (150:2);
        \draw[thick] (0,0) -- (210:1) --+ (150:1) -- (150:2);
        \draw [thick, ->] (150:2) arc (180:210:1.5);
        \fill (150:2) circle (0.07);
        \draw [very thick, ->] (0,0) -- (0,1);
        \node at (150:1) {$i$};
        \node[rotate = 90] at (-0.25,0.5) {\scriptsize{$(i,j)$}};
\end{tikzpicture}
    \caption{the 2-complex of the abstract letter $i$}
    \label{2-complex of i}
\end{figure}

A vertex of $\partial i$ is \textit{marked} if there exists a face $f$ of $\tilde D$ labeled by $i$ such that the corresponding vertex is distinguished. Note that the starting point of $\partial i$ is marked. Marked vertices divide the loop $\partial i$ into segments, called \textit{elementary segments}. 

Consequently, an elementary segment is a sequence of abstract letters $(i,j)(i,j+1)\dots(i,j+t)$ such that, if a path $e_j\dots e_{j+t}$ on $\tilde D$ is decorated by $(i,j)\dots(i,j+t)$, then it passes by no distinguished points except for its endpoints.

\begin{lem}\label{elementary segments} Let $(i,j)\dots(i,j+t)$ be an elementary segment of $(\tilde D,p)$. The abstract letters $(i,j),\dots,(i,j+t)$ are either all free-to-fill, or all semi-free-to-fill, or all not free-to-fill.
\end{lem}
\begin{proof} We shall check that if the vertex between two consecutive abstract letters $(i,j)$ and $(i,j+1)$ is not marked, then they are of the same type. 

Recall that if an edge $\{e_1,e_1\moinsun\}$ is decorated by $(i,j)$ from the face $\{f,f\moinsun\}$, then there is an edge $\{e_2,e_2\moinsun\}$ next to $\{e_1,e_1\moinsun\}$, decorated by $(i,j+1)$ from the same face $\{f,f\moinsun\}$. Assume that the vertex between $(i,j)$ and $(i,j+1)$ is not marked so that the vertex between $\{e_1,e_1\moinsun\}$ and $\{e_2,e_2\moinsun\}$ is not distinguished.

We suppose by contradiction that $(i,j)$ and $(i,j+1)$ are not of the same type. There are $3^2-3=6$ cases, grouped into three cases.

\begin{enumerate} [\textup{case} 1.]
    \item $(i,j)$ is semi-free-to-fill and $(i,j+1)$ is free-to-fill, or inversely: 
    
    Recall that if $(i,j)$ is semi-free-to-fill in the abstract distortion diagram $(\tilde D,p)$, then it decorates an undirected edge $\{e_1,e_1\moinsun\}$ on $p$. As $(i,j+1)$ is free-to-fill, the edge $\{e_2,e_2\moinsun\}$ decorated by $(i,j+1)$ from the same face is not on $p$ (see \textbf{Fig. \ref{case 1}}). So the vertex between $\{e_1,e_1\moinsun\}$ and $\{e_2,e_2\moinsun\}$ is distinguished, contradiction.
    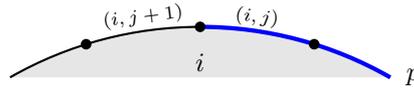
\begin{figure}[h]
        \centering
        \begin{tikzpicture}
            \fill[gray!20] (0,0) arc (60:120:5);
            \draw[thick] (0,0) arc (60:120:5);
            \draw[ultra thick, blue] (0,0) arc (60:90:5);
            \fill ($(-2.5,5)+sqrt(3)*(0,-2.5)$) circle (0.07);
            \fill ($(-1,0)-sqrt(3)*(0,2.5)+sqrt(91)*(0,0.5)$) circle (0.07);
            \fill ($(-4,0)-sqrt(3)*(0,2.5)+sqrt(91)*(0,0.5)$) circle (0.07);
            \node at (0.3,0) {$p$};
            \node[rotate = -7] at (-1.75,0.8) {\scriptsize{$(i,j)$}};
            \node[rotate = 7] at (-3.25,0.8) {\scriptsize{$(i,j+1)$}};
            \node at (-2.5,0.2) {$i$};
        \end{tikzpicture}
        \caption{case 1 of Lemma \ref{elementary segments}}
        \label{case 1}
    \end{figure}
    
    \item $(i,j)$ is not free-to-fill, and $(i,j+1)$ is free-to-fill or semi-free-to-fill: 
    
    By definition, there is an edge $\{e_1,e_1\moinsun\}$ decorated by $(i,j)$ having a smaller decoration $(i',j')<(i,j)$ (see \textbf{Fig. \ref{case 2}}). Let $\{f,f\moinsun\}$, $\{f',{f'}\moinsun\}$ be the faces attached by $\{e_1,e_1\moinsun\}$ such that $f$ is labeled by $i$ and $f'$ is labeled by $i'$.

    Let $\{e_2,e_2\moinsun\}$ be the edge next to $\{e_1,e_1\moinsun\}$, decorated by $(i,j+1)$ from the face $\{f,f\moinsun\}$. As the vertex between $e_1$ and $e_2$ is not distinguished, $\{e_2,e_2\moinsun\}$ is attached to the face $\{f',{f'}\moinsun\}$. It is then decorated by $(i',j'+1)$ or $(i',j'-1)$ from $\{f',{f'}\moinsun\}$. Because $(i,j+1)$ is free-to-fill, we have $(i,j+1)<(i',j'+1)$ or $(i,j+1)<(i',j'-1)$. Both are impossible because $(i,j) > (i',j')$. 
    \begin{figure}[h]
    \centering
        \begin{tikzpicture}
            \fill[gray!20] (2.5,0) ellipse (2.5 and 1);
            \draw[thick] (0,0) -- (5,0);
            \fill (1,0) circle (0.07);
            \fill (2.5,0) circle (0.07);
            \fill (4,0) circle (0.07);
            \node at (1.75,0.2) {\scriptsize{$(i,j)$}};
            \node at (1.75,-0.2) {\scriptsize{$(i',j')$}};
            \node at (3.25,0.2) {\scriptsize{$(i,j+1)$}};
            \node at (2.5,0.5) {$i$};
            \node at (2.5,-0.5) {$i'$};
        \end{tikzpicture}
            \caption{case 2 of Lemma \ref{elementary segments}}
        \label{case 2}
    \end{figure}
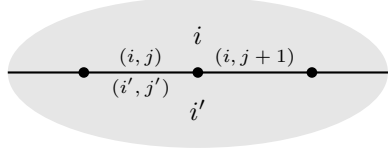
 
    \item $(i,j)$ is free-to-fill or semi-free-to-fill, and $(i,j+1)$ is not free-to-fill: 
    
    There is an edge $\{e_1,e_1\moinsun\}$ decorated by $(i,j+1)$ having a smaller decoration $(i',j')<(i,j+1)$ (\textbf{Fig. \ref{case 3}}). By the same argument of case 2, $(i,j)<(i',j'+1)$ or $(i,j)<(i',j'-1)$. The second one is obviously impossible. If the first one held, then $(i',j') < (i,j+1) < (i',j'+2)$, so $(i',j') = (i,j)$, and there was an edge decorated by $(i,j)$ and $(i,j+1)$ with opposite directions. The canonical function $\phi :\tilde X \to X$ gives $\phi(i,j+1)=\phi(i,j)\moinsun$, which is impossible because $r_i = \phi(i,1)\dots\phi(i,\ell_i)$ should be a reduced word.
        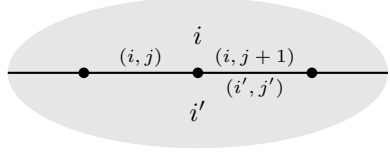
\begin{figure}[h]
    \centering
        \begin{tikzpicture}
            \fill[gray!20] (2.5,0) ellipse (2.5 and 1);
            \draw[thick] (0,0) -- (5,0);
            \fill (1,0) circle (0.07);
            \fill (2.5,0) circle (0.07);
            \fill (4,0) circle (0.07);
            \node at (1.75,0.2) {\scriptsize{$(i,j)$}};
            \node at (3.25,-0.2) {\scriptsize{$(i',j')$}};
            \node at (3.25,0.2) {\scriptsize{$(i,j+1)$}};
            \node at (2.5,0.5) {$i$};
            \node at (2.5,-0.5) {$i'$};
        \end{tikzpicture}
                \caption{case 3 of Lemma \ref{elementary segments}}
        \label{case 3}
    \end{figure}   
\end{enumerate}
\end{proof}

\begin{lem}\label{number of segments} Let $(\tilde D,p)$ be an abstract distortion diagram with no isolated edges.
\begin{enumerate}[$(i)$]
    \item The number of distinguished vertices of $(\tilde D,p)$ is at most $3|\tilde D|$.
    \item The number of elementary segments of an abstract letter $i$ is at most $3|\tilde D|^2$.
\end{enumerate}
\end{lem}
\begin{proof} The underlying 1-complex of $\tilde D$ is a graph of first Betti number $|\tilde D|$ without isolated edges. By Lemma \ref{number of vertices and maximal arcs} there are at most $2(|\tilde D|-1)$ vertices of degree $\geq 3$. We add $k\leq |\tilde D|$ starting points and $2$ endpoints of $p$, there are at most $3|\tilde D|$ distinguished vertices on $(\tilde D,p)$.

The number of faces of $\tilde D$ labeled by $i$ is at most $|\tilde D|$. Every face brings at most $3|\tilde D|$ marked vertices to $\partial i$, so there are at most $3|\tilde D|^2$ marked vertices on $\partial i$.
\end{proof}\quad

Recall that for any abstract relator $i$ of an abstract diagram, we denote $\eta_i$ as the number of free-to-fill abstract letters of $i$ and $\eta'_i$ as the number of semi-free-to-fill abstract letters of $i$.

\begin{lem}\label{number of fillings} Let $X$ be a set of $m\geq 2$ generators. Let $1\leq r\leq m-1$ be an integer. Let $\Gamma$ be a $X$-labeled graph with $b_1(\Gamma)= r$. Let $(\tilde D,p)$ be a reduced abstract distortion diagram with no isolated edges and with $k$ abstract relators. 
\[|N_\ell(\tilde D,p,\Gamma)| \leq \left(\frac{2m}{2m-1}\right)^k (2|\Gamma|)^{3|\tilde D|^2k}(2m-1)^{\sum_{i=1}^{k}\eta_i}(2r-1)^{\sum_{i=1}^{k}\eta'_i}.\]
\end{lem}
\begin{proof} We fill the abstract letters of $\tilde D$ in lexicographic order. We shall prove that if the abstract relators $1,\dots,i-1$ are filled, then there are at most \[\left(\frac{2m}{2m-1}\right) (2|\Gamma|)^{3|\tilde D|^2}(2m-1)^{\eta_i}(2r-1)^{\eta'_i}\]
ways to fill the $i$-th abstract relator.

By Lemma \ref{elementary segments}, we fill elementary segments of $i$ in order. Let $u$ be an elementary segment of $i$. If $u$ is free-to-fill, then there are at most $(2m-1)^{|u|}$ ways to fill $u$ (or at most $2m(2m-1)^{|u|-1}$ ways if $u$ is the first segment of $i$). If $u$ is semi-free-to-fill, then there are at most $2|\Gamma|(2r-1)^{|u|}$ ways to fill $u$ by lemma \ref{number of readable words on a graph}. If $u$ is not free-to-fill, there is only one choice.

The sum of the lengths of free-to-fill segments is $\eta_i$, and the sum of the lengths of semi-free-to-fill segments is $\eta_i'$. As the number of semi-free-to-fill segments is at most $3|\tilde D|^2$ (Lemma \ref{number of segments}), there are at most $2m(2m-1)^{\eta_i-1} (2|\Gamma|)^{3|\tilde D|^2}(2r-1)^{\eta'_i}$ ways to fill the abstract relator $i$.
\end{proof}
\section{Freiheitssatz for random groups}

Recall that $B_\ell$ is the set of cyclically reduced words of $X^\pm = \{x_1^\pm,\dots, x_m^\pm\}$ of length at most $\ell$, and that $|B_\ell|=(2m-1)^{\ell+o(\ell)}$. For $2\leq r\leq m$, the set of cyclically reduced words on $X_r^\pm = \{x_1^\pm,\dots, x_r^\pm\}$ of length at most $\ell$ is of cardinality $(2r-1)^{\ell+o(\ell)}$, so its density in $B_\ell$ is $\log_{2m-1}(2r-1)$. For the case $r=1$, the set of cyclically reduced words on $\{x_1^\pm\}$ is of cardinality $2\ell$, hence with density $0$ in $B_\ell$.

\begin{nota} Let $m\geq 2$ and $1\leq r\leq m-1$ be integers. In this section, we denote
\[c_r = \log_{2m-1}(2r-1)\]
and 
\[d_r = \min\left\{\frac{1}{2}, 1-c_r\right\}.\]
\end{nota}

\subsection{The phase transition}

For any $1\leq r\leq m-1$, we exhibit a phase transition at density $d_r$, characterizing the freeness of certain $r$-generated subgroups in random groups with $m$ generators.

\begin{thm}\label{Freiheitssatz} Let $m\geq 2$ and $1\leq r\leq m-1$ be integers. Let $(G_\ell(m,d))$ be a sequence of random groups at density $d\in[0,1]$.
\begin{enumerate}
    \item If $d>d_r$, then a.a.s., $G_\ell(m,d)$ is generated by $x_1,\dots, x_r$ (or by any subset of $X$ of cardinality $r$).
    \item If $d<d_r$, then a.a.s. for every reduced $X$-labeled graph $\Gamma$ with $b_1(\Gamma)\leq r$ and $|\Gamma|\leq \frac{d_r-d}{5}\ell$, the map $\tilde \Gamma\to\Cay(G_\ell,X)$ is a $\frac{10}{d_r-d}$-bi-Lipschitz embedding.
    
    In addition, a.s.s. every subgroup of $G_\ell(m,d)$ generated by a reduced $X$-labeled graph $\Gamma$ with $b_1(\Gamma)\leq r$ and $|\Gamma|\leq \frac{d_r-d}{5}\ell$ is a free group of rank $r$.
    
    In particular, a.a.s. $x_1,\dots, x_r$ freely generate a free subgroup of $G_\ell(m,d)$.    
\end{enumerate}
\end{thm}\quad

\begin{rem}\label{case r=1} Let us discuss the special case that $r=1$ of Theorem \ref{Freiheitssatz}.
\begin{enumerate}[(a)]
    \item Since $c_1 = \log_{2m-1}(2\cdot 1-1) = 0$, we always have $d_1 = \min\{1/2, 1-c_1\} = 1/2$, independent of the number $m$. The critical density coincides to that of Gromov's triviality-hyperbolicity phase transition (Theorem \ref{hyp}). The first assertion would be obvious as a random group at density $d$ is a.a.s. trivial when $d>d_1=1/2$.
    
    \item The consequence of the second assertion gives that when $d<d_1 = 1/2$, a.a.s. every cyclically reduced word of $X^\pm$ of length at most $\frac{1-2d}{10}\ell$ is a torsion free element in $G_\ell(m,d)$. 
    
    It is actually a special case of another known result: for any $m\geq 2$, if $0\leq d<1/2$ then a.a.s. $G_\ell(m,d)$ is torsion free. See \cite{Oll05} V.d for an argument by Ollivier.
\end{enumerate}
\end{rem}\quad

Let us prove the first assertion of Theorem \ref{Freiheitssatz}, which requires only the intersection formulae (Theorem \ref{random random intersection} and Theorem \ref{random fixed intersection}).

\begin{proof}[Proof of Theorem \ref{Freiheitssatz}.1] We separate the two cases $d_r=1/2<d$ and $d_r<d<1/2$.
\begin{enumerate}[(a)]
    \item  If $d_r=1/2<d$, then a.a.s. the random group $G_\ell(m,d)$ is trivial by Theorem \ref{hyp}. We admit the convention that the trivial element of a trivial group generates itself, so the set of elements $\{x_1,\dots,x_r\}$ generates the whole group. Note that it is always the case when $r=1$ according to Remark \ref{case r=1} (a).
    
    \item Assume that $d_r < d < 1/2$. Recall the notations $X=\{x_1,\dots,x_m\}$ and $X_r=\{x_1,\dots,x_r\}$.

    Let $A_\ell$ be the set of words of type $x_{r+1}w$ where $w$ is a cyclically reduced word of $X_r^\pm$ of length $\ell-1$. The density of $(A_\ell)$ in $(B_\ell)$ is $c_r$. Indeed, if $2\leq r\leq m$, then the cardinality of $A_\ell$ $(2r-1)^{\ell+o(\ell)}$, so $\log_{|B_\ell|}|A_\ell| \to \log_{2m-1}(2r-1) = c_r$; if $r=1$ (although it would never be concerned as $d_1=1/2$ by Remark \ref{case r=1} (a)), then the cardinality of $A_\ell$ is $2$, so $\log_{|B_\ell|}|A_\ell| \to 0 = c_1$.
    
    Since $d_r<1/2$, we have $d_r = 1-c_r$. Together with the hypothesis $d_r<d$, we obtain $c_r+d>1$. Apply the intersection formula (Theorem \ref{random random intersection}), a.a.s. the intersection $R_\ell\cap A_\ell$ is not empty. Hence, a.a.s. there exists a cyclically reduced word $w_{r+1}$ of $X_r^\pm$ such that $x_{r+1}w_{r+1}\in R_\ell$, which implies $x_{r+1}=_{G_\ell}w_{r+1}$.

    Apply the same argument to the other generators $x_{r+2},\dots, x_m$. A.a.s. there are cyclically words $w_{r+1},\dots, w_m$ of $X_r^\pm$ such that $x_i=_{G_\ell}w_i$ for any $r+1\leq i \leq m$. Hence, a.a.s. every word of $X^\pm$ equals to a word of $X_r^\pm$ in $G_\ell$. 
\end{enumerate}
\end{proof}

To prove the second assertion of Theorem \ref{Freiheitssatz}, we state first a local result, showing that the boundary of local distortion diagrams of a random group at density $d$ satisfy some inequality.

\begin{lem}\label{local diagrams} Let $m\geq 2$ and $1\leq r\leq m-1$ be integers. Let $(G_\ell) = (G_\ell(m,d))$ be a sequence of random groups with density $d<d_r$. For any $K>0$, a.a.s. for every reduced $X$-labeled graph $\Gamma$ with $b_1(\Gamma)\leq r$ and $|\Gamma|\leq\frac{d_r-d}{5}\ell$, every \textbf{disc-like} \textbf{reduced} distortion van Kampen diagram $(D,p)$ of $(G_\ell,\Gamma)$ with $|D|\leq K$ satisfies
\[|p|\leq \left(1-\frac{d_r-d}{5}\right)|\partial D|.\tag{$\star$}\]
\end{lem}
The proof of this lemma is in the next subsection.\\

\begin{proof}[Proof of Theorem \ref{Freiheitssatz}.2 by Lemma \ref{local diagrams}]\quad

Apply Lemma \ref{a.a.s. under condition} on the results of Theorem \ref{hyp}.2, let us work under the conditions that every diagram $D$ of $G_\ell(m,d)$ satisfies $|\partial D|\geq (1-2d)/2|D|\ell$ and that $G_\ell(m,d)$ is $\delta$-hyperbolic with $\delta = \frac{4\ell}{1-2d}$.

Let $\Gamma$ be a reduced $X$-labeled graph with $|\Gamma|\leq \frac{d_r-d}{5}\ell$ and $b_1(\Gamma)\leq r$. Let $\lambda = \frac{5}{d_r-d}$. By the local-global principle of quasi-geodesics in a hyperbolic space (Theorem \ref{local-global-quasi-geodesics}), in order to prove that the image of any geodesic of $\tilde{\Gamma}$ in $\Cay(X,G_\ell)$ is a (global) $2\lambda$-quasi geodesic, we shall prove that every reduced word $u$ readable on $\Gamma$ is a $1000\lambda\delta$-local $\lambda$-quasi-geodesic. 

Let $u$ be a reduced word that is readable on $\Gamma$ with $|u|\leq 1000\lambda\delta$. Let $v$ be a geodesic in $G_\ell$ joining endpoints of the image of $u$ in $G_\ell$. We shall prove that $|u|\leq \lambda |v|$. By van Kampen's lemma (Lemma \ref{van Kampen's lemma}) there exists a diagram $D$ of $G_\ell$ whose boundary word is $uv$. By the isoperimetric inequality (Theorem \ref{hyp}.2) and the fact that $|\partial D|= |u|+|v|\leq 2|u| \leq 2000\lambda\delta$, we have
\[|D|\leq \frac{2|\partial D|}{(1-2d)\ell}\leq \frac{40000}{(1-2d)^2(d_r-d)}.\]

Apply Lemma \ref{local diagrams} with $K = \frac{40000}{(1-2d)^2(d_r-d)}$. If $D$ is disk-like, then by $(\star)$, we have \[|u|\leq \left(1-\frac{d_r-d}{5}\right)(|u|+|v|)\leq \frac{\lambda}{1+\lambda}(|u|+|v|),\]
which implies $|u|\leq \lambda |v|$.

Otherwise, we decompose $D$ into discs and segments. By the same argument of Lemma \ref{fff}, because every disc-like sub-diagram is a distortion diagram satisfying $(\star)$, we still have $|u|\leq \lambda |v|$. Hence, the word $u$ is a $1000\lambda\delta$-local $\lambda$-quasi-geodesic with $\lambda = \frac{5}{d_r-d}$. We conclude that  the map $\tilde \Gamma\to\Cay(G_\ell,X)$ is a $\frac{10}{d_r-d}$-bi-Lipschitz embedding.\\

Let $H$ be a subgroup generated by $\Gamma$. If $H$ were not free, then there would be a reduced loop $p$ of $\Gamma$ whose labeling word is a trivial word in the random group $G_\ell(m,d)$. But this word should be a quasi-geodesic by the previous result, which gives a contradiction. To prove that $x_1,\dots,x_r$ generate freely a free subgroup, one can take $\Gamma$ as the wedge of $r$ cycles of length one labeled by $x_1,\dots,x_r$. 
\end{proof}

\subsection{Local distortion diagrams of a random group}

We will prove Lemma \ref{local diagrams} in this subsection. The proof is similar to Ollivier's proof of the local isoperimetric inequality for random groups (Lemma \ref{local isoperimetric inequality}) in \cite{Oll05} p.86.

Let $(G_\ell) = (G_\ell(m,d)) = (\langle X|R_\ell\rangle)$ be a sequence of random groups with density $d<d_r$. We work first on the fillability of an abstract distortion diagram. Denote \[\varepsilon_d = \frac{d_r-d}{5}.\]

Since the sequence of random subsets $(R_\ell)$ is densable with density $d$, the probability event \[Q_\ell:=\left\{(2m-1)^{(d-\varepsilon_d)\ell} \leq |R_\ell| \leq (2m-1)^{(d+\varepsilon_d)\ell}\right\}\] is a.a.s. true.
\begin{lem}[Fillability of an abstract distortion diagram]\label{fillability of an abstract distortion diagram} Let $K>0$. Let $\Gamma$ be a reduced $X$-labeled graph with $b_1(\Gamma)\leq r$ and $|\Gamma|\leq \varepsilon_d\ell$. Let $(\tilde D,p)$ be a disc-like abstract distortion diagram with $|\tilde D|\leq K$ that satisfies
\[|p| > \left(1-\varepsilon_d\right)|\partial \tilde D|.\]
Then for $\ell$ large enough,
\[\P = \Pr\left((\tilde D,p) \textup{ is fillable by }(G_\ell,\Gamma)\,\middle|\, Q_\ell\right)\leq \ell^{10K^3}(2m-1)^{-2\varepsilon_d\ell}.\]
\end{lem}
\begin{proof} We shall prove the lemma in four steps. We omit ``for $\ell$ large enough'' in every step. Recall that $\alpha_i$ is the number of faces labeled by $i$, $\eta_i$ the number of free-to-fill abstract letters of $i$, and $\eta'_i$ the number of semi-free-to-fill abstract letters of $i$.

\begin{enumerate}[\textbf{Step} 1:]
    \item $\displaystyle{\log_{2m-1}\P \leq \sum_{i=1}^k(\eta_i+c_r\eta'_i+(d-1+2\varepsilon_d)\ell) + 10K^3\log_{2m-1}\ell}.\qquad\quad (1)$
    
    According to Proposition \ref{k elements in a random subset}, if $(r_1,\dots,r_k)$ is a filling of $\tilde D$ by $B_\ell$, then for $\ell$ large enough $\Pr\left(r_1,\dots,r_k\in R_\ell\,\middle|\, Q_\ell\right)\leq (2m-1)^{k(d-1+2\varepsilon_d)\ell}$. Recall that $N_\ell(\tilde D,p,\Gamma)$ is the set of fillings of $(\tilde D,p)$ by $(B_\ell,\Gamma)$.
    
    Apply Lemma \ref{number of fillings} with $|\Gamma|\leq \varepsilon_d\ell$ and $|\tilde D|\leq K$,
    \begin{align*} & \Pr\left((\tilde D,p) \textup{ is fillable by }(G_\ell,\Gamma)\,\middle|\, Q_\ell\right)\\
    \leq & \sum_{(r_1,\dots,r_k)\in N_\ell(\tilde D,p,\Gamma)} \Pr\left(r_1,\dots,r_k\in R_\ell\,\middle|\, Q_\ell\right)\\
    \leq &|N_\ell(\tilde D,p,\Gamma)| (2m-1)^{k(d-1+2\varepsilon_d)\ell}\\
    \leq &\ell^{10K^3}(2m-1)^{\sum_{i=1}^{k}\eta_i}(2r-1)^{\sum_{i=1}^{k}\eta'_i}(2m-1)^{k(d-1+2\varepsilon_d)\ell}.
    \end{align*}
    Hence the inequality $(1)$ by applying $\log_{2m-1}$.
    \item $\displaystyle{|\tilde D| \left(\log_{2m-1}\P-10K^3\log_{2m-1}\ell\right) \leq \sum_{i=1}^k\alpha_i(\eta_i+c_r\eta'_i+(d-1+2\varepsilon_d)\ell)}.\quad (2)$
    
    Let $\tilde D_i$ be the sub-diagram of $\tilde D$ consisting of the faces labeled by the first $i$ abstract relators $1^\pm,\dots,i^\pm$ and the edges attached to them. Apply $(1)$ to $\tilde D_i$, and denote $\P_i$ the probability obtained. We have
    \[\log_{2m-1}\P \leq \log_{2m-1}\P_i \leq \sum_{s=1}^i(\eta_s+c_r\eta'_s+(d-1+2\varepsilon_d)\ell) + 10K^3\log_{2m-1}\ell.\]
    
    Without loss of generality, we assume $\alpha_1\geq \alpha_2\geq\dots\geq \alpha_k$. Note that $\log_{2m-1}\P$ is negative and that $\alpha_1\leq |\tilde D|$. By Abel's summation formula, with convention $\alpha_{k+1} = 0$,
    \begin{align*} &\sum_{i=1}^k \alpha_i\left(c_r\eta'_i+\eta_i + (d-1+2\varepsilon_d)\ell \right) \\ = & \sum_{i=1}^k \left[ (\alpha_i-\alpha_{i+1})\sum_{s=1}^i\left(c_r\eta'_s+\eta_s + (d-1+2\varepsilon_d)\ell \right)\right]\\
    \geq &\sum_{i=1}^k \left[ (\alpha_i-\alpha_{i+1})(\log_{2m-1}\P-10K^3\log_{2m-1}\ell) \right]\\
    \geq &\alpha_1(\log_{2m-1}\P-10K^3\log_{2m-1}\ell)\\
    \geq &|\tilde D|(\log_{2m-1}\P-10K^3\log_{2m-1}\ell).
    \end{align*}
    
    \item $\displaystyle{\log_{2m-1}\P\leq \left(d-\frac{1}{2}+2\varepsilon_d\right) \ell + \left(c_r-\frac{1}{2}+\varepsilon_d\right)\frac{|\partial\tilde D|}{|\tilde D|}+10K^3\log_{2m-1}\ell}.\quad (3)$
    
    Let $\varepsilon_d'>0$ such that $|\overline p|=(1-\varepsilon_d')|\partial \tilde D|$. By hypothesis $\varepsilon_d'<\varepsilon_d$. Because $\tilde D$ is disc-like and the boundary length of every face is $\leq\ell$, the number of undirected edges $|\overline E|$ is less than $\frac{|\tilde D|\ell-|\partial \tilde D|}{2} + |\partial\tilde D|$. Apply Lemma \ref{alphai etai}, we get
    \[\sum_{i=1}^k \alpha_i\eta'_i \leq |\overline p|=(1-\varepsilon_d')|\partial\tilde D|, \]
    \[\sum_{i=1}^k \alpha_i\eta_i \leq |\overline E|-|\overline p|\leq \frac{|\tilde D|\ell}{2} + \left( \varepsilon_d'-\frac{1}{2}\right)|\partial\tilde D|.\]
    Note that $\sum_{i=1}^k \alpha_i = |\tilde D|$. So we have
    \begin{align*} &\sum_{i=1}^k \alpha_i\left(c_r\eta'_i+\eta_i + (d-1+2\varepsilon_d)\ell \right) \\ 
    \leq & c_r(1-\varepsilon_d')|\partial\tilde D| + \frac{|\tilde D|\ell}{2} + \left( \varepsilon_d'-\frac{1}{2}\right)|\partial\tilde D| + (d-1+2\varepsilon_d)|\tilde D|\ell\\
    \leq & \left(d-\frac{1}{2}+2\varepsilon_d\right) |\tilde D|\ell + \left(c_r-\frac{1}{2}+\varepsilon_d\right)|\partial\tilde D|.
    \end{align*}
    
    Combine this inequality with $(2)$, we get $(3)$
    
    \item $\displaystyle{\left(d-\frac{1}{2}+2\varepsilon_d\right) \ell + \left(c_r-\frac{1}{2}+\varepsilon_d\right)\frac{|\partial\tilde D|}{|\tilde D|}}\leq -2\varepsilon_d\ell.\qquad\qquad\qquad\qquad\qquad (4)$
    
    Recall that $c_r = \log_{2m-1}(2r-1).$ Note that $|\partial \tilde D|\leq \ell|\tilde D|$ and that $d = d_r -5\varepsilon_d$. There are two cases:
    \begin{enumerate}
    \item If $c_r\geq\frac{1}{2}$, then $d = 1-c_r-5\varepsilon_d$ and $c_r-\frac{1}{2}+\varepsilon_d \geq 0$, so
    \begin{align*} & \left(d-\frac{1}{2}+2\varepsilon_d\right) |\tilde D|\ell + \left(c_r-\frac{1}{2}+\varepsilon_d\right)|\partial\tilde D|\\
    \leq & \left(d+c_r-1+3\varepsilon_d\right) |\tilde D|\ell\\
    \leq & -2\varepsilon_d |\tilde D|\ell
    \end{align*}
    \item If $c_r < \frac{1}{2}$, then $d = \frac{1}{2}-5\varepsilon_d$ and $c_r-\frac{1}{2}+\varepsilon_d \leq \varepsilon_d$, so
    \begin{align*} & \left(d-\frac{1}{2}+2\varepsilon_d\right) |\tilde D|\ell + \left(c_r-\frac{1}{2}+\varepsilon_d\right)|\partial\tilde D|\\
    \leq & \left(d-\frac{1}{2}+3\varepsilon_d\right) |\tilde D|\ell\\
    \leq & -2\varepsilon_d |\tilde D|\ell
    \end{align*}
    \end{enumerate}
\end{enumerate}
By $(3)$ and $(4)$, for $\ell$ large enough $\displaystyle{\log_{2m-1}(\P)\leq -2\varepsilon_d\ell+10K^3\log_{2m-1}\ell }.$
\end{proof}\quad

By Lemma \ref{number of vertices and maximal arcs} and Lemma \ref{number of topological types}, we have the following two results.

\begin{lem}\label{number of graphs} For $\ell$ large enough, the number of reduced labeled connected graphs $\Gamma$ (with respect to $X$) with $|\Gamma|\leq \varepsilon_d\ell$ and $b_1(\Gamma)\leq r$ is bounded by
\[(2m-1)^{\varepsilon_d\ell}\ell^{3r}.\]\\[-3em]\qed
\end{lem}




\begin{lem}\label{number of abstract diagrams} For $\ell$ large enough, the number of disc-like abstract distortion diagrams $(\tilde D,p)$ with $|\tilde D|\leq K$ and $|\partial f|\leq \ell$ for all faces $f\in F$ is bounded by
\[\ell^{5K}.\]\\[-3em]\qed
\end{lem}\quad




\begin{proof}[Proof of Lemma \ref{local diagrams}] Recall that $\varepsilon_d = \frac{d_r-d}{5}$.

We shall prove that a.a.s. for every reduced $X$-labeled graph $\Gamma$ with $b_1(\Gamma)\leq r$ and $|\Gamma|\leq\varepsilon_d\ell$, every reduced distortion diagram $(D,p)$ of $(G_\ell,\Gamma)$ with $|D|\leq K$ satisfies $|p|\leq \left(1-\varepsilon_d\right)|\partial D|.$

Apply Lemma \ref{fillability of an abstract distortion diagram}, Lemma \ref{number of graphs} and Lemma \ref{number of abstract diagrams}. The probability that there exists a reduced $X$-labeled graph $\Gamma$ with $b_1(\Gamma)\leq r$, $|\Gamma|\leq \varepsilon_d\ell$ and there exists a disc-like reduced abstract distortion diagram $(\tilde D,p)$ with $|\tilde D|\leq K$, $|p|>\left(1-\varepsilon_d\right)|\partial\tilde D|$ such that $(\tilde D,p)$ is fillable by $(G_\ell,\Gamma)$ is bounded by
\[(2m-1)^{\varepsilon_d\ell}\ell^{3r} \times \ell^{5K} \times \ell^{10K^3}(2m-1)^{-2\varepsilon_d\ell} = (2m-1)^{-\varepsilon_d\ell+O(\log\ell)}.\]

So the probability that there exists a reduced $X$-labeled graph $\Gamma$ with $b_1(\Gamma)\leq r$, $|\Gamma|\leq \varepsilon_d\ell$ and there exists a disc-like reduced distortion diagram $(D,p)$ of $(G_\ell,\Gamma)$ with $|D|\leq K$ that satisfies $|p|>\left(1-\varepsilon_d\right)|\partial\tilde D|$ is bounded by

\[(2m-1)^{-\varepsilon_d\ell+O(\log\ell)},\]
which goes to $0$ when $\ell$ goes to infinity.
\end{proof}\quad

This completes the proof of Theorem \ref{Freiheitssatz}.\\

\begin{rem}\label{why bounded graphs} In this proof, the number of edges of graphs gives an exponential factor $\sim (2m-1)^{|\Gamma|}$ by Lemma \ref{number of graphs}, which is the main reason that by our current method we need the condition $|\Gamma| \leq \frac{d_r-d}{5}\ell$  in Theorem \ref{Freiheitssatz}.2.

Note that in \textbf{Step 4} of Lemma \ref{fillability of an abstract distortion diagram}, the numbers ``$+2\varepsilon_d$'' and ``$+\varepsilon_d$'' in the left hand side can be replaced by arbitrary small positive numbers, so the estimation ``$-2\varepsilon_d$'' in the right hand side can be arbitrary close to $-5\varepsilon_d = -(d_r-d)$. Hence, the condition $|\Gamma| \leq \frac{d_r-d}{5}\ell$ in Theorem \ref{Freiheitssatz}.2 can be pushed up to $|\Gamma| \leq (d_r-d-s)\ell$ for any small $s>0$ in our current method.
\end{rem}\quad\\

\begin{rem}\label{unbounded graphs} While our method for proving Theorem \ref{Freiheitssatz}.2 does not work for $X$-labeled graphs with $|\Gamma|\geq (d_r-d)\ell$, it does not mean that the conclusion fails for these graphs.

For instance, according to \cite{Gro87} 5.3.A, since a.a.s. the random group $G_\ell(m,d)$ is $\frac{4\ell}{1-2d}$ hyperbolic, subgroups generated by long enough words are free. More precisely, we know that for any $d<1/2$, a.a.s. there exists a large enough $K = K(m,d)$ such that, if $\Gamma$ is a $X$-labeled graph with $b_1(\Gamma) = r$ whose maximal arcs are longer than $K\ell$, then a subgroup generated by $\Gamma$ is free and quasi-convex. 

Note that $K\ell$ is much larger than the hyperbolicity constant $\delta = \frac{4}{1-2d}\ell > \ell$, which is still much larger than the bound $(d-d_r)\ell < \ell$ of Theorem \ref{Freiheitssatz}.2.
\end{rem}\quad\\

\end{document}